\documentclass[reqno]{amsart}

\usepackage{todonotes}
\usepackage{graphicx,float}
\usepackage{amsthm,amsmath,amssymb,amstext,amsfonts,color}

\usepackage{fullpage}

\usepackage{etoolbox}
\patchcmd{\thebibliography}{\section*{\refname}}{}{}{}

\usepackage[all]{xy} 
\usepackage{ytableau}

\newtheorem{Theorem}{Theorem}[section]
\newtheorem{Conjecture}[Theorem]{Conjecture}
\newtheorem{Proposition}[Theorem]{Proposition}
\newtheorem{Corollary}[Theorem]{Corollary}
\newtheorem{Lemma}[Theorem]{Lemma}
\theoremstyle{remark}

\newtheorem{ex}{Example}

\newcommand{\field}[1]{\mathbb{#1}}
\newcommand{\N}{\field{N}}

\newcommand{\K}{\textbf{K}}

\numberwithin{equation}{section}

\allowdisplaybreaks

\title{New companions to the Andrews--Gordon identities motivated by commutative algebra}

\author{Pooneh Afsharijoo}
\address{Universit\'e Paris Cit\'e, Sorbonne Universit\'e, CNRS, Institut de
Math\'ematiques de Jussieu-Paris Rive Gauche, F-75013 Paris, France}
\email{pooneh.afshari@gmail.com}

\author{Jehanne Dousse}
\address{Université de Genève, Section de Mathématiques, 7-9 rue du Conseil-Général, CH-1205 Genève, Switzerland}
\email{jehanne.dousse@unige.ch}

\author{Fr\'ed\'eric Jouhet}
\address{Univ Lyon, Université Claude Bernard Lyon 1, UMR5208, Institut Camille Jordan, F-69622 Villeurbanne, France}
\email{jouhet@math.univ-lyon1.fr}

\author{Hussein Mourtada}
\address{Universit\'e  Paris Cit\'e, Sorbonne Universit\'e, CNRS, Institut de
Math\'ematiques de Jussieu-Paris Rive Gauche, F-75013 Paris, France}
\email{hussein.mourtada@imj-prg.fr}


\begin{document}

\begin{abstract}
We give a proof of a recent conjectural partition identity due to the first author, which was discovered in the framework of commutative algebra. This result gives rise to new companions to the famous Andrews--Gordon identities. Our tools involve graded quotient rings, Durfee squares and rectangles for integer partitions, and $q$-series identities.
\end{abstract}

\maketitle

\section{Introduction}

A partition of a positive integer $n$ is a non-increasing sequence of positive integers $\lambda=(\lambda_1,\lambda_2,\dots,\lambda_s)$  such that $\lambda_1+\lambda_2+\cdots+\lambda_s=n$. The integers $\lambda_i$ are called the parts of $\lambda$ and $s$ is its length. For example, the partitions of $4$ are $(4)$, $(3,1)$, $(2,2)$, $(2,1,1)$, and $(1,1,1,1)$.

An important research direction in the theory of partitions, which goes back at least to Euler, is the study of partition identities.  Such an identity states that for every positive integer $n$, the numbers of partitions of $n$ satisfying two different types of constraints are equal. When the generating functions (with respect to $n$ and the usual variable $q$) of the two corresponding sets of partitions are considered, such an identity becomes equivalent to an equality of $q$-series. On the other hand, some $q$-series identities were proved before discovering the corresponding interpretation in terms of partitions. The most famous such instances are the Rogers--Ramanujan identities, which are the following formal power series identities:
\begin{eqnarray}    
\sum_{k=0}^\infty\frac{q^{k^2}}{(1-q)\cdots(1-q^k)}&=&\prod_{n\geq0}\frac{1}{(1-q^{5n+1})(1-q^{5n+4})},\label{rr1}\\
\sum_{k=0}^\infty\frac{q^{k^2+k}}{(1-q)\cdots(1-q^k)}&=&\prod_{n\geq0}\frac{1}{(1-q^{5n+2})(1-q^{5n+3})},\label{rr2}
\end{eqnarray}
where empty products obtained with $k=0$ on the left-hand sides are taken to be $1$. The Rogers--Ramanujan identities are among the most fascinating and deep, as they are connected to combinatorics, statistical mechanics, number theory, representation theory, and algebraic geometry (see for instance~\cite{A2, Bax, BIS, BMS1, GIS, GOW,MM}). After their discovery, they were interpreted combinatorially by MacMahon \cite{MacMahon} and Schur \cite{Schur}, giving rise to the following partition identities.
\begin{Theorem}[Rogers--Ramanujan identities, combinatorial version]
\label{th:RRcomb}
Let $n$ be a nonnegative integer and set $i\in\{1,2\}$. Denote by $T_{2,i}(n)$ the number of partitions of $n$ such that the difference between consecutive parts is at least $2$ and the part $1$ appears at most $i-1$ times. Let $E_{2,i}(n)$ be the number of partitions of $n$ into parts congruent to $\pm (2+i) \mod 5$. Then we have
$$T_{2,i}(n)=E_{2,i}(n).$$
\end{Theorem}

A famous family of partition identities, which generalizes the Rogers--Ramanujan identities and plays a central role in this article, is due to Gordon~\cite{Go}.

\begin{Theorem}[Gordon's identities]\label{th:Gordon}
Let $r$ and $i$ be integers such that $r\geq 2$ and $1\leq i \leq r$. Let $\mathcal{T}_{r,i}$ be the set of partitions $\lambda=(\lambda_1,\lambda_2,\dots, \lambda_s)$ where $\lambda_{j}-\lambda_{j+r-1} \geq 2$ for all $j$, and at most $i-1$ of the parts $\lambda_j$ are equal to $1$. Let $\mathcal{E}_{r,i}$ be the set of partitions whose parts are not congruent to $0,\pm i \mod (2r+1)$. Let $n$ be a nonnegative integer, and let $T_{r,i}(n)$ (respectively $E_{r,i}(n)$) denote the number of partitions of $n$ which belong to $\mathcal{T}_{r,i}$ (respectively $\mathcal{E}_{r,i}$). Then we have 
$$T_{r,i}(n)=E_{r,i}(n).$$
\end{Theorem}

The Rogers--Ramanujan identities~\eqref{rr1} and~\eqref{rr2} correspond to the cases $r=i=2$ and $r=i+1=2$ in Theorem \ref{th:Gordon}, respectively.

Our main goal in this article is to prove new companions to the Gordon identities; in particular, we will consider a new set of partitions that we will call $\mathcal{C}_{r,i}$ and prove that for all nonnegative integers $n$, the number $C_{r,i}(n)$ of partitions of  $n$ belonging to  $\mathcal{C}_{r,i}$ is equal to $T_{r,i}(n)$ and $E_{r,i}(n)$.
 This settles positively a conjecture made by the first author in~\cite{A}. 

Before introducing the set $\mathcal{C}_{r,i}$, we begin by explaining the origin of this conjecture, which also allows us to give an algebro-geometric setting that will be useful in some parts of the proof (even though we also give a purely combinatorial proof of these parts). For that, consider the ring of polynomials 
$$\mathcal{R}=\K[x_i,i\geq 1]$$
with countably many variables over a field $\K$ of characteristic $0$. We give $\mathcal{R}$ a structure of graded ring 
by assigning to $x_i$ the weight $i$; this means that $\mathcal{R}=\oplus_{n\geq 0}R_n$ where $R_0=\K$ and $R_n$ is the $\K$-vector space with a basis given by the monomials $x_{i_1}\cdots x_{i_s}$ (we can assume that $i_1\geq i_2\cdots\geq i_s>0$) such that $i_1+\cdots+i_s=n$. 

In particular, there is a trivial bijection between monomials of weight $n$ and partitions of $n$: the monomial $x_{i_1}\cdots x_{i_s}$ of weight $n$ is in bijection with the partition $\lambda=(i_1,\ldots,i_s)$ of $n$. \textit{In the remainder of this paper, we always denote by $x_{\lambda}$ the monomial associated to the partition $\lambda$.}

 Therefore the Hilbert--Poincar\'e series $HP_{\mathcal{R}}$ of $\mathcal{R}$ is given by 
$$HP_{\mathcal{R}}(q):=\sum_{n\geq 0}\mbox{dim}_\K R_nq^n=\sum_{n\geq 0} p(n)q^n,$$
where $p(n)$ is the number of partitions of $n.$

Equivalently, these partitions can be interpreted as functions on the space of arcs centered at the origin of the affine line  $\textbf{A}_\K^1=\mathrm{Spec}\K[t]$  (here $t$ is a formal variable) defined over the field $\K$. Recall that the space of arcs of a variety is the moduli space which parametrizes formal curves traced on this variety. In other words, the space of arcs of a $\K$-variety $X$ is the scheme whose $\K$-points are in bijection with the morphisms $\mathrm{Spec}\K[[t]]\longrightarrow X$. It follows from~\cite{BMS1,BMS2} that functions on the space of arcs centered at fat points are in correspondence with the partitions in $\mathcal{T}_{r,i}$. More precisely, for any integer $r\geq2$, the ring of global sections of the space of arcs centered at a fat point $\mathrm{Spec}\K[t]/(t^r)$ is the quotient ring
$\mathcal{R}/[x_1^r]$, where 
$[x_1^r]$ is the differential ideal generated by $x_1^r$ and its iterated derivative with respect to the derivation $D$ defined by $D(x_j):=x_{j+1}$.



Thus we have
$$[x_1^r]=(x_1^r,rx_1^{r-1}x_2,r(r-1)x_1^{r-2}x_2^2+rx_1^{r-1}x_3,\ldots).$$
It follows from~\cite{BMS1} that for integers $1\leq i\leq r-1$, the leading ideal of $\mathcal{J}_{r,i}:=(x_1^i,[x_1^r])$ with respect to  the ``weighted reverse lexicographical order'', that is  the ideal generated by the leading monomials of all the elements in $\mathcal{J}_{r,i}$, is
$$J_{r,i}=(x_1^i,x_k^{r-s}x_{k+1}^s;k\geq 1; s=0,\ldots,r-1).$$ 

It is well-known that the Hilbert--Poincar\'e series of a graded ring quotiented by an ideal $I$ is equal to the Hilbert--Poincar\'e series of the ring quotiented by the leading ideal of $I$ with respect to any monomial ordering which is compatible with the grading. It is important here to mention that the leading ideal is in general not generated by the leading monomials of a system of generators; but a system of generators of an ideal $I$ such that the leading monomials of its members generate the leading ideal of $I$ is called a Gr\"obner basis. In particular, we have $HP_{\mathcal{R}/\mathcal{J}_{r,i}}(q)=HP_{\mathcal{R}/J_{r,i}}(q)$. But a quick examination of   
$$\frac{\mathcal{R}}{J_{r,i}}$$ 
is sufficient to see that its monomials (these are the monomials in $\mathcal{R}$ which do not belong to $J_{r,i}$) correspond exactly to the partitions in $\mathcal{T}_{r,i}$.   
 Hence we have 

$$HP_{\mathcal{R}/\mathcal{J}_{r,i}}(q)=HP_{\mathcal{R}/J_{r,i}}(q)=1+\sum_{n\geq 1}T_{r,i}(n)q^n.$$

In~\cite{A}, the first author tried to compute the leading ideal of $\mathcal{J}_{r,i}$ with respect to the weighted lexicographical order and she predicted that it is equal to the ideal $I_{r,i}\subset \textbf{K}[x_1,x_2,\ldots]$ generated by $x_1^i$ and the monomials of the following form:

$$\underbrace{x_{n_{1,1}}}_{\text {first block}} \underbrace{ x_{n_{2,1}}\cdots x_{n_{2,f_{r,i}(2)}}}_{\text {second block}} \underbrace{ x_{n_{3,1}} \cdots x_{n_{3,f_{r,i}(3)}}}_{\text{third block}} \cdots \underbrace{ x_{n_{r,1}}\cdots x_{n_{r,f_{r,i}(r)}}}_{\text{$r$-th block}};$$

\noindent where 

$$f_{r,i}(j):= \begin{cases}
1 &\text{ if } j=1, \\
n_{j-1,f_{r,i}(j-1)} &\text{ if } 2\leq j \leq i, \\
n_{j-1,f_{r,i}(j-1)}-1 &\text{ if } i+1 \leq j \leq r.\\
\end{cases}$$

We do not know how to prove that this is actually the leading ideal because it involves the computation of a Gr\"obner basis of $\mathcal{J}_{r,i}$ with respect to the weighted lexicographical order, which is out of reach for the moment. Such a Gr\"obner basis is of course infinite but also does not seem to have finiteness properties even for $r=2$: for instance it is not differentially finite (which means that it cannot be generated by a finite number of elements in $\mathcal{R}$ and by their iterated derivatives)~\cite{AM}, contrary to the case where one considers the weighted reverse lexicographical order~\cite{BMS1,BMS2}. This led the first author to introduce the above-mentioned set $\mathcal{C}_{r,i}$ of partitions which correspond to the monomials in the quotient ring $\mathcal{R}/I_{r,i}$ and that we are now ready to describe.

Given an integer $r\geq 2$, we define  for $1 \leq i \leq r$ the $(i,\ell)$-new part of $\lambda=(\lambda_1,\dots,\lambda_s)$ as follows:

$$p_{i,\ell}(\lambda):=\begin{cases}
\lambda_s &\text{ if } \ell=1, \\
\lambda_{s-\sum_{j=1}^{\ell-1}p_{i,j}(\lambda)} &\text{ if }  2\leq \ell \leq i, \\
\lambda_{s+\ell-i-\sum_{j=1}^{\ell-1}p_{i,j}(\lambda)} &\text{ if } i <\ell \leq r-1,
\end{cases}$$ 
where $\lambda_j=0$ for $j\leq 0,$ and if $p_{i,\ell}(\lambda)=0$ then $p_{i,j}(\lambda)=0$ for $j> \ell$. We denote the number of all non-zero $(i,\ell)$-new parts of $\lambda$ by $N_{r,i}(\lambda)$. 
In~\cite{A}, the first author conjectured the following.

\begin{Conjecture}\label{conj:Pooneh_original}
Let $r \geq 2$ and $1 \leq i \leq r$ be two integers. Let  $\mathcal{C}_{r,i}$ be the set of partitions of the form $\lambda=(\lambda_1,\dots,\lambda_s),$ such that at most $i-1$ of the parts are equal to $1$ and either $N_{r,i}(\lambda)<r-1$, or $N_{r,i}(\lambda)=r-1$ and $s\leq \sum_{j=1}^{r-1} p_{i,j}(\lambda)-(r-i)$. Let $n$ be a nonnegative integer, and denote by $C_{r,i}(n)$ the number of partitions of $n$ which belong to $\mathcal{C}_{r,i}$. Then we have
$$C_{r,i}(n)=T_{r,i}(n)=E_{r,i}(n).$$
\end{Conjecture}

Our main result is a proof of this conjecture.

\begin{Theorem}\label{th:conjtrue}
Conjecture~\ref{conj:Pooneh_original} is true.
\end{Theorem}

In order to prove Theorem~\ref{th:conjtrue}, we will define several combinatorial objects related to Durfee dissections. Not only do we prove Theorem~\ref{th:conjtrue}, but we also obtain the equality with two other sets of partitions. The new types of Durfee dissections, the corresponding new sets of partitions and the way they are connected to Conjecture~\ref{conj:Pooneh_original} are described in Section~\ref{sec:outline}.

Our combinatorial tools to prove Theorem~\ref{th:conjtrue} are inspired by a companion to the Gordon identities due to Andrews, called the Andrews--Gordon identities~\cite{A74}.
Before stating them, recall some standard notations for $q$-series which can be found in~\cite{GR}. 
The $q$-shifted factorial is defined by
\begin{equation*}
(a)_\infty\equiv (a;q)_\infty:=\prod_{j\geq 0}(1-aq^j)\;\;\;\;\mbox{and}\;\;\;\;(a)_k\equiv (a;q)_k:=\frac{(a;q)_\infty}{(aq^k;q)_\infty},
\end{equation*}
where $k$ is any integer.
Since the base $q$ is often the same throughout this paper,
it may be readily omitted (in notation, writing $(a)_k$ instead of $(a;q)_k$, etc.) which will not lead to any confusion. For brevity, write
\begin{equation*}
(a_1,\ldots,a_m;q)_k:=(a_1)_k\cdots(a_m)_k,
\end{equation*}
where $k$ is an integer or infinity. The $q$-binomial coefficient is defined as follows:
$$\left[{n\atop k}\right]_q:=\frac{(q)_n}{(q)_k(q)_{n-k}},$$
and we notice that by definition $\left[{n\atop k}\right]_q=0$ if $k<0$ or $k>n$. It is the generating function for partitions with largest part $\leq k$ and number of parts $\leq n-k$, or equivalently partitions whose Young diagram fits inside a $k \times (n-k)$ rectangle.

In~\cite{A74}, Andrews rewrote the Gordon identities from Theorem~\ref{th:Gordon} in a combinatorial form, and later extended them as a $q$-series identity. The latter can be stated as follows.
\begin{Theorem}[Andrews--Gordon identities]\label{th:AGseries}
Let $r \geq 2$ and $1 \leq i \leq r$ be two integers. We have
\begin{equation}\label{eq:AGri}
\sum_{n_1\geq\dots\geq n_{r-1}\geq0}\frac{q^{n_1^2+\dots+n_{r-1}^2+n_{i}+\dots+n_{r-1}}}{(q)_{n_1-n_2}\dots(q)_{n_{r-2}-n_{r-1}}(q)_{n_{r-1}}}=\frac{(q^{2r+1},q^{i},q^{2r-i+1};q^{2r+1})_\infty}{(q)_\infty}.
\end{equation}
\end{Theorem}
Note that~\eqref{rr1} is obtained from~\eqref{eq:AGri} by taking $r=i=2$ while~\eqref{rr2} is obtained from~\eqref{eq:AGri} by taking $r=2$ and  $i=1$.
The left-hand side of \eqref{eq:AGri} can be rewritten as follows
\begin{equation}
\label{eq:AG_qbinom}
\sum_{n_1\geq\dots\geq n_{r-1}\geq0}\frac{q^{n_1^2+\dots+n_{r-1}^2+n_{i}+\dots+n_{r-1}}}{(q)_{n_1}} \left[{n_1\atop n_2}\right]_q \cdots \left[{n_{r-2}\atop n_{r-1}}\right]_q,
\end{equation}
which led Andrews~\cite{A79} to find a simple combinatorial interpretation of Theorem~\ref{th:AGseries} in terms of Durfee squares and Durfee dissections that we will describe in Section~\ref{sec:outline}: we will recall Andrews' set $\mathcal{A}_{r,i}$ of partitions whose generating function is given by the left-hand side of~\eqref{eq:AGri} (or equivalently~\eqref{eq:AG_qbinom}). On the other hand, the right-hand side of~\eqref{eq:AGri} is clearly the generating function for partitions in $\mathcal{E}_{r,i}$.

Inspired by this, and introducing new kinds of Durfee dissections, we will compute the generating function for the partitions in Conjecture~\ref{conj:Pooneh_original}. Thus, as will be explained in Section~\ref{sec:outline}, we will reduce the proof of the conjecture to proving the following identity, valid for all integers $r>0$ and $0\leq i\leq r-1$:
\begin{equation}\label{AGP}
\sum_{s_1\geq\dots\geq s_{r-1}\geq0}\frac{q^{s_1^2+\dots+s_{r-1}^2-s_1-\dots-s_i}(1-q^{s_i})}{(q)_{s_1-s_2}\dots(q)_{s_{r-2}-s_{r-1}}(q)_{s_{r-1}}}=\frac{(q^{2r+1},q^{r-i},q^{r+i+1};q^{2r+1})_\infty}{(q)_\infty},
\end{equation}
where we take the convention that when $i=0$, the left-hand side (where $s_0$ is not well-defined) is simply
$$\sum_{s_1\geq\dots\geq s_{r-1}\geq0}\frac{q^{s_1^2+\dots+s_{r-1}^2}}{(q)_{s_1-s_2}\dots(q)_{s_{r-2}-s_{r-1}}(q)_{s_{r-1}}}.$$

One sees that~\eqref{rr1} is obtained from~\eqref{AGP} by taking $r=2$ and $i=0$ while~\eqref{rr2} is obtained from~\eqref{AGP} by taking $r=2$, $i=1$ and shifting the integer $s_1$ to $s_1+1$ in the summation.

Note that the role played by $i$ in~\eqref{eq:AGri} is now played by $r-i$, but to fit with classical techniques and notations regarding $q$-series that will follow, we chose to keep this change of index. 

Actually, for $i=0$ the above formula is exactly the instance $i=r$ of~\eqref{eq:AGri}, while for $i>0$, \eqref{AGP} is a consequence of the following result due to Bressoud~\cite[(3.3)]{Br80}.
\begin{Theorem}[Bressoud]\label{thm:bressoud3.3}
For all integers $r>0$ and $0\leq i\leq r-1$, we have:
\begin{equation}\label{Br3.3}
\sum_{s_1\geq\dots\geq s_{r-1}\geq0}\frac{q^{s_1^2+\dots+s_{r-1}^2-s_1-\dots-s_i}}{(q)_{s_1-s_2}\dots(q)_{s_{r-2}-s_{r-1}}(q)_{s_{r-1}}}=\sum_{k=0}^{i}\frac{(q^{2r+1},q^{r-i+k},q^{r+i-k+1};q^{2r+1})_\infty}{(q)_\infty}.
\end{equation}
\end{Theorem}
Note that there seems to be a mistake in Bressoud's formula (3.3), in which $\pm(k-r+i)$ (in his notation) should be changed to $\pm(k-r+i+1)$. Moreover, denoting by $S_i(q)$ the left-hand side of~\eqref{Br3.3} (the dependence on $r$ is omitted), we immediately see that for $i>0$, the left-hand side of~\eqref{AGP} is equal to
 $$S_i(q)-S_{i-1}(q)=\sum_{k=0}^{i}\frac{(q^{2r+1},q^{r-i+k},q^{r+i-k+1};q^{2r+1})_\infty}{(q)_\infty}-\sum_{k=1}^{i}\frac{(q^{2r+1},q^{r-i+k},q^{r+i-k+1};q^{2r+1})_\infty}{(q)_\infty},$$
 which is telescoping and yields the right-hand side of~\eqref{AGP}.

A classical approach to obtain and prove identities like~\eqref{eq:AGri} and~\eqref{AGP} is the Bailey lemma, originally found by Bailey~\cite{Ba} and whose iterative strength was later highlighted by Andrews~\cite{A1, A2, AAR} through the so-called Bailey chain. As will be recalled later, this is an efficient way to prove some instances of the Andrews--Gordon identities in Theorem~\ref{th:AGseries}, but it fails for general $r,i$. One indeed needs a more general tool, namely the Bailey lattice, first described in~\cite{AAB}. Alternatively, as shown in~\cite[Section~3]{ASW}, one can combine the Bailey lemma with tricky calculations, therefore bypassing the Bailey lattice (see also~\cite{BIS}, where it is explained how changing the base also avoids the need of the Bailey lattice). In Section~\ref{sec:Bailey} we will prove the corrected version of Bressoud's result given in Theorem~\ref{thm:bressoud3.3} by using the Bailey lattice, therefore providing a different approach from Bressoud's original one in~\cite{Br80}.\\


This paper is organized as follows. In Section~\ref{sec:outline} we provide all the combinatorial objects and tools which are necessary for our proof. We will then be able to introduce new combinatorial sets of partitions involving Durfee squares and rectangles, namely $\mathcal{B}_{r,i}$ (which is nothing but a reformulation of the set $\mathcal{C}_{r,i}$) and $\mathcal{D}_{r,i}$. This will enable us to restate Conjecture~\ref{conj:Pooneh_original} in a way that is easier to handle combinatorially (see Conjecture~\ref{conj:Pooneh_combinatorial}). We will also see that the generating series for the set $\mathcal{D}_{r,r-i}$ is the left-hand side of~\eqref{AGP}. In Section~\ref{sec:combi_Durf_Bot} we prove combinatorially that $\mathcal{B}_{r,i}=\mathcal{D}_{r,i}$, and alternatively we prove algebraically in Section~\ref{sec:alg_Durf_Bot} that $\mathcal{C}_{r,i}=\mathcal{D}_{r,i}$. Next we prove Theorem~\ref{thm:bressoud3.3} through the Bailey lattice in Section~\ref{sec:Bailey}, which immediately implies~\eqref{AGP}, showing itself that $\mathcal{D}_{r,r-i}$ and $\mathcal{E}_{r,r-i}$ have the same generating series, therefore finishing the proof of Theorem~\ref{th:conjtrue}.
In Section~\ref{sec:bij} we give a direct bijection between partitions of $n$ in $\mathcal{D}_{r,r-1}$ and in $\mathcal{A}_{r,r-1}$, which, in complement to the easy cases $i=1$ and $i=r$, is the only situation for which we have a completely combinatorial proof of the conjecture. Finally, we conclude in Section~\ref{sec:final} by a list of open problems which arise from our study.

\section{New Durfee dissections and proof strategy}
\label{sec:outline}

Let us now turn to the combinatorial objects and dissections which will be used in our proof, before stating our main theorem in a combinatorial form.

We start by recalling the Durfee dissection which was defined by Andrews in his combinatorial interpretation of the Andrews--Gordon identities~\eqref{eq:AGri} in~\cite{A79}. We use a slightly different terminology than his, which will help avoid any confusion with our new types of dissections. Define the Durfee square of a partition $\lambda$ to be the largest square of size $k \times k$ fitting in the top-left corner of the Young diagram of $\lambda$.
In Figure~\ref{fig:durfee_dissec}, $A_1$ is the Durfee square of the partition.

Similarly we can define its vertical Durfee rectangle to be the largest vertical rectangle of size $(k-1) \times k$, i.e. with $k-1$ columns and $k$ rows, fitting in the top-left corner of its Young diagram.

It is possible to define successive Durfee squares/rectangles by drawing the first Durfee square/rectangle, and then drawing the Durfee square/rectangle of the partition restricted to the parts below it, and repeating the process until the row below a square/rectangle is empty.
For convenience in our future proofs, we take the convention that we can still draw Durfee squares/rectangles after exiting the partition, but that they are empty. When 
we choose that the first $i-1$ Durfee squares/rectangles are squares, and that all the following ones are rectangles, the sequence of non-empty Durfee squares/rectangles in $\lambda$ is uniquely defined and is called the (vertical) $(i-1)$-Durfee dissection of $\lambda$.
We denote the successive Durfee squares (resp. rectangles) by $A_1, \dots, A_{i-1}$ (resp. $A'_{i}, A'_{i+1}, \dots$). 

Figure \ref{fig:durfee_dissec} shows the vertical $2$-Durfee dissection of a partition (the last rectangle $A'_4$ is of size $0 \times 1$, and all rectangles below are empty).
\begin{figure}[H]
\includegraphics[width=0.4\textwidth]{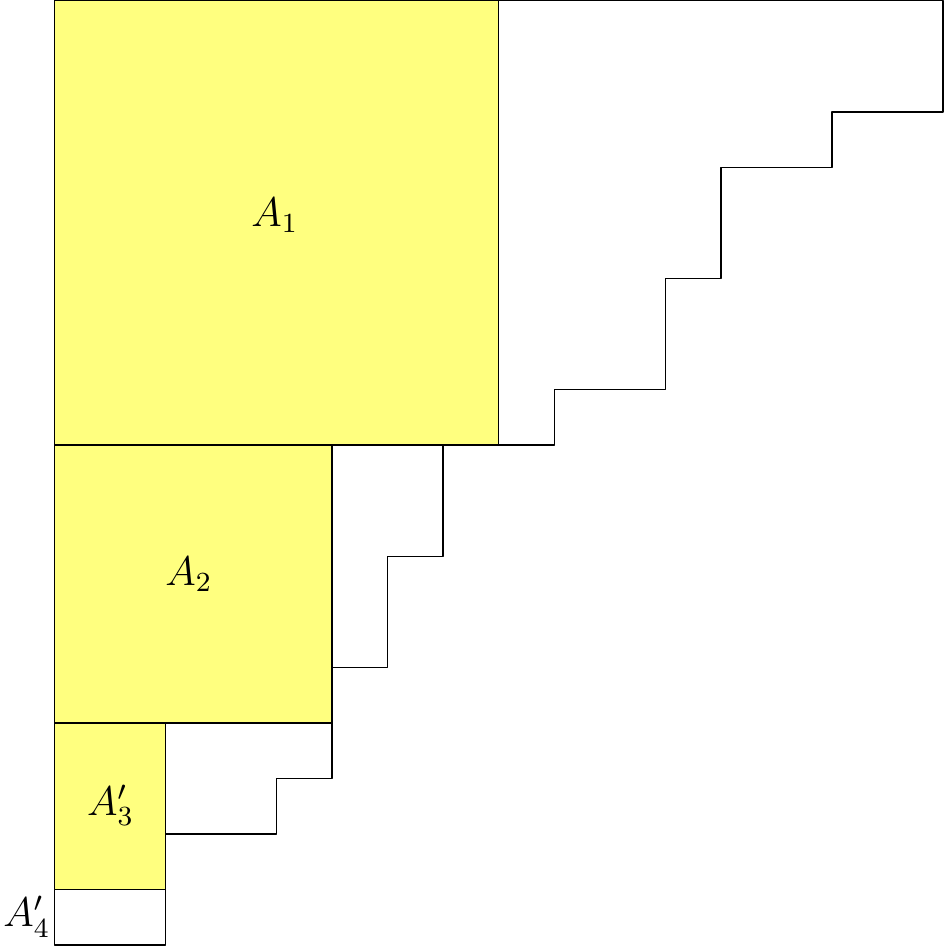}
\caption{Vertical $2$-Durfee dissection}
\label{fig:durfee_dissec}
\end{figure}

Recall from the introduction that for integers $r \geq 2$ and $1 \leq i \leq r$,  $\mathcal{E}_{r,i}$ is the set of partitions whose parts are not congruent to $0,\pm i \mod (2r+1)$, and that for all nonnegative integers $n$, $E_{r,i}(n)$ is the number of partitions of $n$ which belong $\mathcal{E}_{r,i}$. Andrews' combinatorial version of the Andrews--Gordon identities is the following.

\begin{Theorem}[Andrews]
\label{th:AGcombi}
Let $r \geq 2$ and $1 \leq i \leq r$ be two integers. Let $\mathcal{A}_{r,i}$ be the set of partitions such that in their vertical $(i-1)$-Durfee dissection, all vertical Durfee rectangles below $A'_{r-1}$ are empty, and such that the last row of each non-empty Durfee rectangle is actually a part of the partition.  For all nonnegative integers $n$, denote by $A_{r,i}(n)$ the number of partitions of $n$ which belong to $\mathcal{A}_{r,i}$. Then we have
$$A_{r,i}(n)=E_{r,i}(n).$$
\end{Theorem}

Figure \ref{fig:andrews} shows a partition in $\mathcal{A}_{5,3}$. Indeed, in its vertical $2$-Durfee dissection, all the vertical Durfee rectangles below $A'_4$ are empty, and the last row of each non-empty vertical rectangle (represented in orange) is actually a part of the partition. The crosses represent boxes which, by definition, do not belong to the partition.
\begin{figure}
\includegraphics[width=0.55\textwidth]{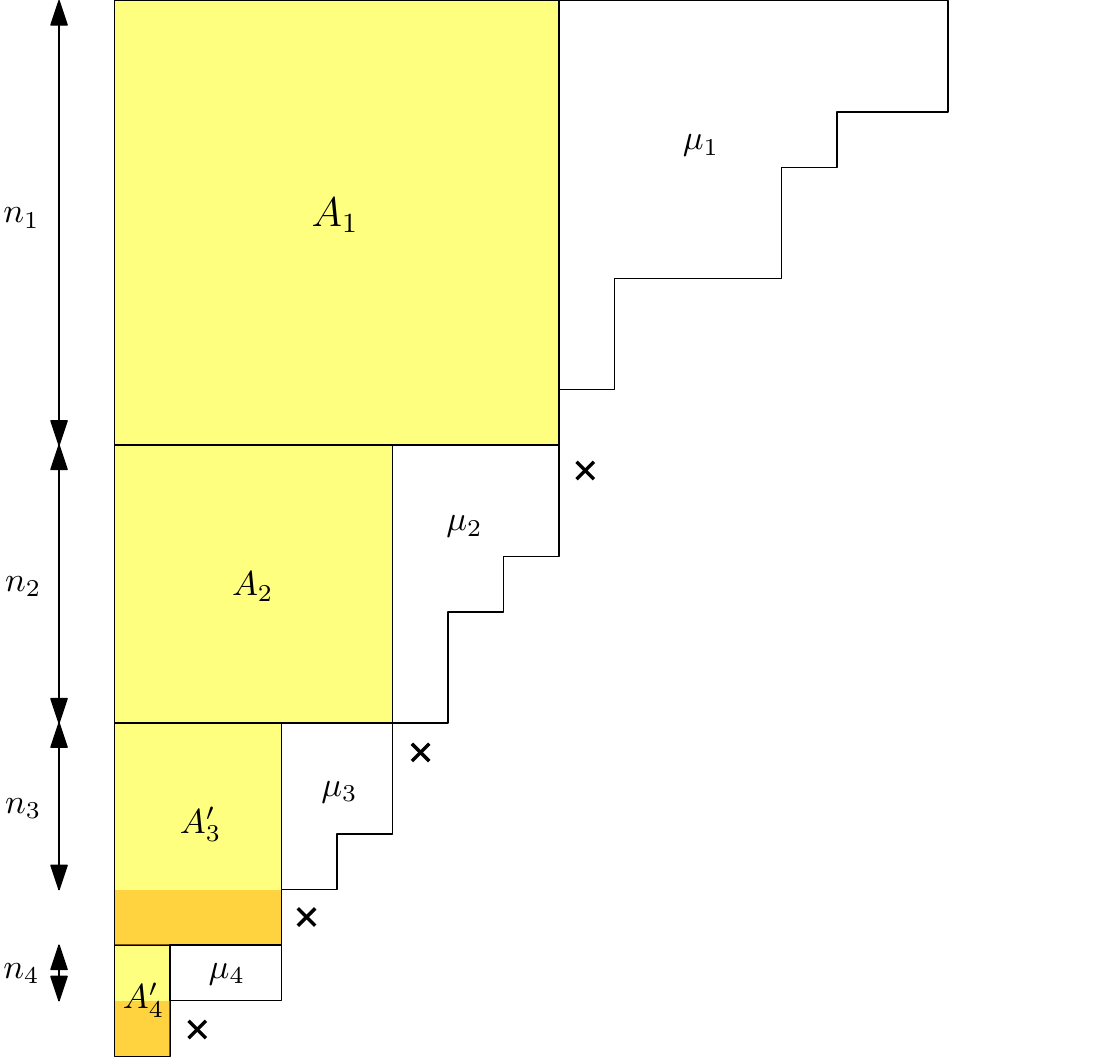}
\caption{Vertical $2$-Durfee dissection of a partition in $\mathcal{A}_{5,3}$}
\label{fig:andrews}
\end{figure}

Partitions in $\mathcal{A}_{r,i}$ are generated by~\eqref{eq:AG_qbinom}, that is the left-hand side of the $q$-series version for Andrews--Gordon identities given in Theorem~\ref{th:AGseries}. Indeed, denote by $\mu_j$ the partition to the right of the Durfee square (resp. rectangle) $A_j$ (resp. $A_j'$) and by $n_j$ the size of the smallest side of $A_j$ (resp. $A'_j$). See Figure~\ref{fig:andrews} for an illustration of these notations.  Then $\mu_1$ is a partition into at most $n_1$ parts, generated by $1/(q)_{n_1}$. Then for all $2 \leq j \leq r-1$, the partition $\mu_j$ has to fit inside a $(n_{j-1}-n_{j}) \times n_j$ rectangle, which is generated by the $q$-binomial coefficient $\left[{n_{j-1}\atop n_{j}}\right]_q$. Finally, the Durfee squares $A_1, \dots, A_{i-1}$ are generated by $q^{n_1^2}, \dots , q^{n_{i-1}^2}$ and the vertical Durfee rectangles $A'_i, \dots, A'_{r-1}$ are generated by $q^{n_i^2+n_i}, \dots , q^{n_{r-1}^2+n_{r-1}}$.\\

To prove Conjecture \ref{conj:Pooneh_original}, we first reformulate it in a more combinatorial way.

We define the bottom square (resp. bottom rectangle) of a partition $\lambda=(\lambda_1,\dots,\lambda_s)$ to be the square of size $\lambda_s \times \lambda_s$ (resp. the horizontal rectangle of size $\lambda_s \times (\lambda_s-1)$) whose bottom coincides with the bottom of the Young diagram of $\lambda$.
In Figure~\ref{fig:bottom_squares}, $B_1$ is the bottom square of the partition.

Just like for Durfee squares, we can define successive bottom squares/rectangles by drawing the first bottom square/rectangle, and then drawing the bottom square/rectangle of the partition restricted to the parts above it, and repeating the process until the row above a square/rectangle is empty. 
For convenience, we take the convention that we can still draw bottom squares/rectangles after exiting the partition, but that they are empty. We also allow bottom rectangles of size $1 \times 0$ (this can appear if the smallest part of the partition is a $1$). 
When we choose that the first $i-1$ bottom squares/rectangles are squares, and that all the following ones are rectangles, the sequence of non-empty bottom squares/rectangles in $\lambda$ is uniquely defined and we call it the \textit{$(i-1)$-bottom dissection} of $\lambda$.
We denote the successive bottom squares (resp. rectangles) by $B_1, \dots, B_{i-1}$ (resp. $B'_{i}, B'_{i+1}, \dots$). 

Figure~\ref{fig:bottom_squares} shows the successive bottom squares/rectangles of a partition, with two successive bottom squares (the bottom rectangles above $B'_4$ are empty).
\begin{figure}[H]
\includegraphics[width=0.4\textwidth]{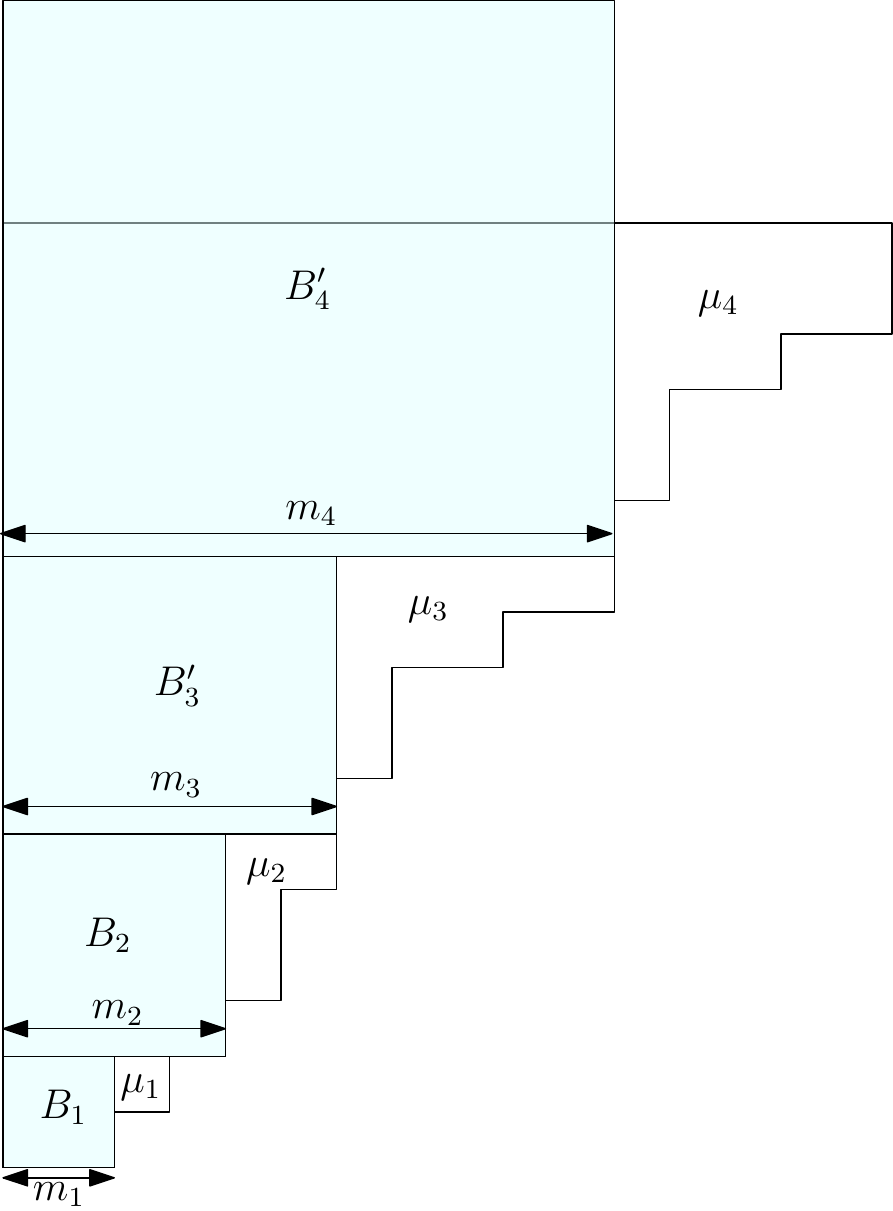}
\caption{The $2$-bottom dissection of a partition}
\label{fig:bottom_squares}
\end{figure}

Let $\mathcal{B}_{r,i}$ be the set of partitions such that in their $(i-1)$-bottom dissection, all bottom rectangles above $B'_{r-1}$ are empty.
In other words, if one draws $i-1$ successive bottom squares $B_1, \dots, B_{i-1}$ followed by $r-i$ bottom rectangles $B'_{i}, \dots,  B'_{r-1}$, then the row above $B'_{r-1}$ is empty. Denote by $B_{r,i}(n)$ the number of partitions of $n$ which belong to $\mathcal{B}_{r,i}$.
For example, the partition in Figure~\ref{fig:bottom_squares} belongs to $\mathcal{B}_{5,3}$ but not to $\mathcal{B}_{4,3}$. 

By definition of bottom squares/rectangles, for all $1 \leq i \leq r$, we have  
$$\mathcal{B}_{r,i}= \mathcal{C}_{r,i},$$ 
so Conjecture~\ref{conj:Pooneh_original} can be reformulated as follows.

\begin{Conjecture}[Reformulation of Conjecture~\ref{conj:Pooneh_original}]\label{conj:Pooneh_combinatorial}
Let $r \geq 2$ and $1 \leq i \leq r$ be two integers. Then for all nonnegative integers $n$, we have
$$B_{r,i}(n)=A_{r,i}(n)=E_{r,i}(n).$$
\end{Conjecture}

However, whereas it is possible to compute the generating function for partitions in $\mathcal{B}_{r,i}$, it does not seem easy to show directly that it equals the generating function \eqref{eq:AG_qbinom} for partitions in $\mathcal{A}_{r,i}$. Indeed, one can proceed as for the above generating series for $\mathcal{A}_{r,i}$, but from bottom to top instead of top to bottom: denote by $\mu_j$ the partition to the right of the bottom square (resp. rectangle) $B_j$ (resp. $B'_j$), by $m_j$ the size of the largest side of $B_j$ (resp. $B'_j$), let $B_k$ (resp. $B'_k$) with $1 \leq k \leq r-1$ be the last non-empty bottom square (resp. rectangle), and let $m$ be the distance between the bottom of $B_{k}$ (resp. $B'_k$) and the top of our partition (note that $\mu_{k}$ has length $\leq m-1$). The generating function for partitions in $\mathcal{B}_{r,i}$ takes therefore  the following form:

\begin{multline}
 \label{eq:B_qbinom}
1+\sum_{k=1}^{i-1}\sum_{m_{k}\geq\dots\geq m_1\geq1}\left(\sum_{m=1}^{m_k}\frac{q^{mm_k}}{(q)_{m-1}}\right)q^{\sum_{\ell=1}^{k-1}m_{\ell}^2} \prod_{\ell=1}^{k-1} \left[{m_{\ell+1}-1\atop m_{\ell}-1}\right]_q \\
+\sum_{k=i}^{r-1}\sum_{m_{k}\geq\dots\geq m_1\geq1}\left(\sum_{m=1}^{m_k-1}\frac{q^{mm_k}}{(q)_{m-1}}\right)q^{\sum_{\ell=1}^{k-1}m_{\ell}^2-\sum_{\ell=i}^{k-1}m_{\ell}} \prod_{\ell=1}^{i-1} \left[{m_{\ell+1}-1\atop m_{\ell}-1}\right]_q  \prod_{\ell=i}^{k-1} \left[{m_{\ell+1}-2\atop m_{\ell}-2}\right]_q.
\end{multline}
Simplifying the $q$-binomial coefficients, this can be rewritten as
\begin{multline}\label{eq:B_qserie}
1+\sum_{k=1}^{i-1} \sum_{m_{k}\geq\dots\geq m_1\geq1} \left(\sum_{m=1}^{m_k}\frac{q^{mm_k}}{(q)_{m-1}}\right) q^{\sum_{\ell=1}^{k-1}m_{\ell}^2} \frac{(q)_{m_k-1}}{(q)_{m_1-1}} \prod_{\ell=1}^{k-1} \frac{1}{(q)_{m_{\ell+1}-m_{\ell}}}
\\+\sum_{k=i}^{r-1}\sum_{m_{k}\geq\dots\geq m_1\geq 1}\left(\sum_{m=1}^{m_{k}-1}\frac{q^{mm_{k}}}{(q)_{m-1}}\right)q^{\sum_{\ell=1}^{k-1}m_{\ell}^2-\sum_{\ell=i}^{k-1}m_{\ell}} (1-q^{m_i-1}) \frac{(q)_{m_k-2}}{(q)_{m_1-1}} \prod_{\ell=1}^{k-1} \frac{1}{(q)_{m_{\ell+1}-m_{\ell}}}.
\end{multline}

Thus our first task is to show that the conjecture is equivalent to a conjecture involving successive Durfee squares and rectangles. But in contrast to Andrews' dissection~\cite{A79}, our Durfee rectangles will be horizontal, we will start with rectangles and finish with squares, and we will not have the restriction that the last row of Durfee rectangles have to actually be parts of the partition. 

Define the horizontal Durfee rectangle of a partition $\lambda=(\lambda_1, \dots,\lambda_s)$ to be the largest horizontal rectangle of size $k \times (k-1)$ fitting in the top-left corner of the Young diagram of $\lambda$.
\emph{From now on, when we mention a Durfee rectangle without further precision, we mean horizontal Durfee rectangle.}
In Figure~\ref{fig:bottom_durfee_squares}, $D'_1$ is the Durfee rectangle of the partition.

As we did for bottom squares/rectangles, we can define successive Durfee squares/rectangles by drawing the first Durfee square/rectangle, and then drawing the Durfee square/rectangle of the partition restricted to the parts below, and repeating the process until the row below a square/rectangle is empty.
Again, we take the convention that we can still draw Durfee squares/rectangles after exiting the partition, but that they are empty. We also allow Durfee rectangles of size $1 \times 0$, which are not considered to be empty (this can happen when there is a part $1$). 
When we choose that the first $k$ Durfee squares/rectangles are rectangles, and that the following are all squares, the sequence of non-empty Durfee squares/rectangles in $\lambda$ is uniquely defined and is called the \textit{$k$-Durfee dissection} of $\lambda$.
We denote the successive Durfee rectangles (resp. squares) by $D'_1, \dots, D'_{k}$ (resp. $D_{k+1}, D_{k+2}, \dots$). 

Figure~\ref{fig:bottom_durfee_squares} shows the $2$-Durfee dissection of the same partition as before (the Durfee squares below $D_4$ are all empty).

\begin{figure}[H]
\includegraphics[width=1\textwidth]{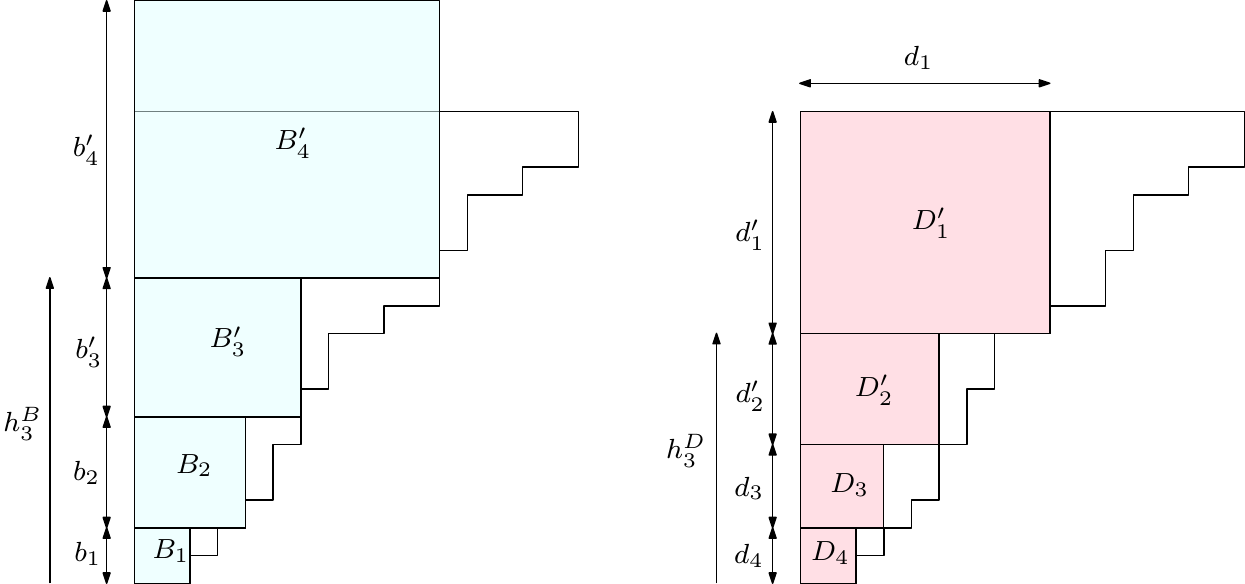}
\caption{The $2$-bottom dissection and $2$-Durfee dissection of the same partition}
\label{fig:bottom_durfee_squares}
\end{figure}

Now define $\mathcal{D}_{r,i}$ to be the set of partitions such that in their $(r-i)$-Durfee dissection, all Durfee squares below $D_{r-1}$ are empty.
In other words, if one draws $r-i$ (horizontal) Durfee rectangles $D'_1, \dots, D'_{r-i}$ followed by $i-1$ Durfee squares $D_{r-i+1}, \dots,  D_{r-1}$, then the row below $D_{r-1}$ is empty.
For example, the partition in Figure \ref{fig:bottom_durfee_squares} belongs to $\mathcal{D}_{5,3}$ but not to $\mathcal{D}_{4,3}$.

Note that Figure~\ref{fig:bottom_durfee_squares} shows a particular partition which belongs both to $\mathcal{B}_{5,3}$ and to $\mathcal{D}_{5,3}$. We show that this is a general phenomenon and that the following holds.

\begin{Theorem}\label{th:Durfee_Bottom}
Let $r \geq 2$ and $1 \leq i \leq r$ be two integers. We have
$$\mathcal{B}_{r,i}=\mathcal{D}_{r,i}.$$
\end{Theorem}
We prove Theorem~\ref{th:Durfee_Bottom} combinatorially in Section~\ref{sec:combi_Durf_Bot} and algebraically in Section~\ref{sec:alg_Durf_Bot} (the algebraic proof does not use the set $\mathcal{B}_{r,i},$ it proves directly that $\mathcal{C}_{r,i}=\mathcal{D}_{r,i}).$   

Now doing the same as we did above for $\mathcal{A}_{r,i}$ and  $\mathcal{B}_{r,i}$, we can compute  the generating function for partitions in $\mathcal{D}_{r,i}=\mathcal{B}_{r,i}$ to derive  a simpler form than in~\eqref{eq:B_qbinom} (we got rid of the sum over $m$); we indeed get
\begin{equation}\label{eq:D_qbinom}
\sum_{d_1\geq\dots\geq d_{r-1}\geq0}\frac{q^{d_1^2+\dots+d_{r-1}^2-d_1-\dots-d_{r-i}}}{(q)_{d_1-1}} \left[{d_1-1\atop d_2-1}\right]_q \cdots \left[{d_{r-i-1}-1\atop d_{r-i}-1}\right]_q \times\left[{d_{r-i}\atop d_{r-i+1}}\right]_q \cdots \left[{d_{r-2}\atop d_{r-1}}\right]_q.
\end{equation}

Here, for all $1 \leq j \leq r-1$, $d_j$ represents the size of the larger side of the Durfee square (resp. rectangle) $D_j$ (resp. $D'_j$).

Simplifying the $q$-binomial coefficients, the generating function in~\eqref{eq:D_qbinom} can be rewritten as:
\begin{equation}\label{eq:D_qserie}
\sum_{d_1\geq\dots\geq d_{r-1}\geq0}\frac{q^{d_1^2+\dots+d_{r-1}^2-d_1-\dots-d_{r-i}}}{(q)_{d_1-d_2}\dots(q)_{d_{r-2}-d_{r-1}}(q)_{d_{r-1}}}(1-q^{d_{r-i}}),
\end{equation}
which is again simpler than~\eqref{eq:B_qserie}. It should be possible to prove a weaker version of Theorem~\ref{th:Durfee_Bottom} analytically by showing that~\eqref{eq:B_qserie} and~\eqref{eq:D_qserie} are equal; however it did not seem obvious to us how it could be done, so we looked for a combinatorial proof instead, which also has the advantage of giving more insight on the different types of dissections.

Thus what is left to do in order to prove the conjecture is showing that~\eqref{eq:D_qserie} equals the generating function for partitions in $\mathcal{A}_{r,i}$ or $\mathcal{E}_{r,i}$.

In the case of only squares $(i=r)$, the partitions in $\mathcal{A}_{r,r}$ and $\mathcal{D}_{r,r}$ are the same by definition. In the case of only rectangles ($i=1$), there is a simple bijection between $\mathcal{A}_{r,1}$ and $\mathcal{D}_{r,1}$ by rotating the horizontal Durfee rectangles in  $\mathcal{A}_{r,1}$ by $90$ degrees and thus obtaining partitions in $\mathcal{D}_{r,1}$, and vice versa. However this simple bijection does not work for other values of $i$, as some problems can appear at the transition between squares and rectangles, and because Andrews' Durfee dissection for $\mathcal{A}_{r,i}$ starts with squares and ends with rectangles while ours for $\mathcal{D}_{r,i}$ does the contrary. 

We found a more complicated bijection in the particular case $i=r-1$, given in Section~\ref{sec:bij}. The question of finding a bijection between  $\mathcal{A}_{r,i}$ and $\mathcal{D}_{r,i}$ in the general case still eludes us.

Therefore our proof of Conjecture~\ref{conj:Pooneh_combinatorial} will actually consist in showing that the generating function~\eqref{eq:D_qserie} of $\mathcal{D}_{r,i}$ equals the infinite product which is the generating function for $\mathcal{E}_{r,i}$. 
This follows from~\eqref{AGP}: indeed, the right-hand side of~\eqref{AGP} is the generating series for $\mathcal{E}_{r,r-i}$ obtained by taking $r-i$ instead of $i$ in the right-hand side of~\eqref{eq:AGri}, while the left-hand side corresponds to the generating series of $\mathcal{D}_{r,r-i}$ obtained by taking $r-i$ instead of $i$ in~\eqref{eq:D_qserie}. 

This shows that Conjecture~\ref{conj:Pooneh_combinatorial} (and therefore Conjecture~\ref{conj:Pooneh_original}) is an immediate consequence of~\eqref{AGP} (and therefore Theorem~\ref{thm:bressoud3.3}) and Theorem~\ref{th:Durfee_Bottom}.

\section{Combinatorial connection between the conjecture and Durfee squares/rectangles}
\label{sec:combi_Durf_Bot}
In this section, we prove Theorem \ref{th:Durfee_Bottom} combinatorially, i.e. we prove that the partitions in $\mathcal{B}_{r,i}$ and $\mathcal{D}_{r,i}$ are exactly the same.

First, introduce some more notation which can be seen in Figure \ref{fig:bottom_durfee_squares}.
For $1 \leq j \leq r-1$, let $\mu_j$ denote the sub-partition to the right of the $j$-th Durfee square/rectangle.
For all $1 \leq j \leq i-1$, let $b_j$ (resp. $d_j$) denote the size of the bottom square (resp. Durfee square) $B_j$ (resp. $D_j$). For all $i \leq j \leq r-1$, let $b'_j$ (resp. $d'_j$) denote the size of the smaller side of the bottom rectangle (resp. Durfee rectangle) $B'_j$ (resp. $D'_j$). The primes are there to remind us whether we are in the case of a square or a rectangle.

Moreover, assuming that the bottom of the smallest part of the partition is of height $0$, we denote for all $1 \leq j \leq r-1$ by $h^B_j$ (resp. $h^D_j$) the height of the top of the $j$-th bottom square/rectangle (resp. Durfee square/rectangle) starting from the bottom. Therefore, we have for all $1 \leq j \leq r-1$,
\begin{align*}
h^B_j= \tilde{b}_1 +  \tilde{b}_2 + \dots + \tilde{b}_j,\\
h^D_j= \tilde{d}_{r-1} + \tilde{d}_{r-2} + \dots + \tilde{d}_{r-j}.
\end{align*}
where $\tilde{b}_k$ (resp. $\tilde{d}_k$) equals $b_k$ (resp. $d_k$) when $1 \leq k \leq i-1$ and $b'_k$ (resp. $d'_k$) when $i \leq k \leq r-1$.

We also take the convention that $h^D_r=+\infty.$

\medskip

We also mention a simple lemma from set theory, which will simplify our proofs in this section.

\begin{Lemma}
\label{lem:disjoint_union}
Let $S$ be a set which can be written in two different ways as a disjoint union
$$S = \bigsqcup_{n \in \N} S_n = \bigsqcup_{n \in \N} S'_n,$$
such that for all $n \in \N$, $S_n \subseteq S'_n$. Then for all $n \in \N$, $S_n = S'_n$.
\end{Lemma}

\subsection{The case with only squares: $\mathcal{B}_{r,r}=\mathcal{D}_{r,r}$}
We are now ready to start our proof. Let us begin by showing that when there are only squares (case $i=r$), $\mathcal{B}_{r,r}=\mathcal{D}_{r,r}$. This is the simplest case, and it will be useful in the proof of the general case.

Figure \ref{fig:all_squares} shows the successive Durfee squares of a partition, together with its successive bottom squares, and the corresponding heights.

\begin{figure}[H]
\includegraphics[width=0.9\textwidth]{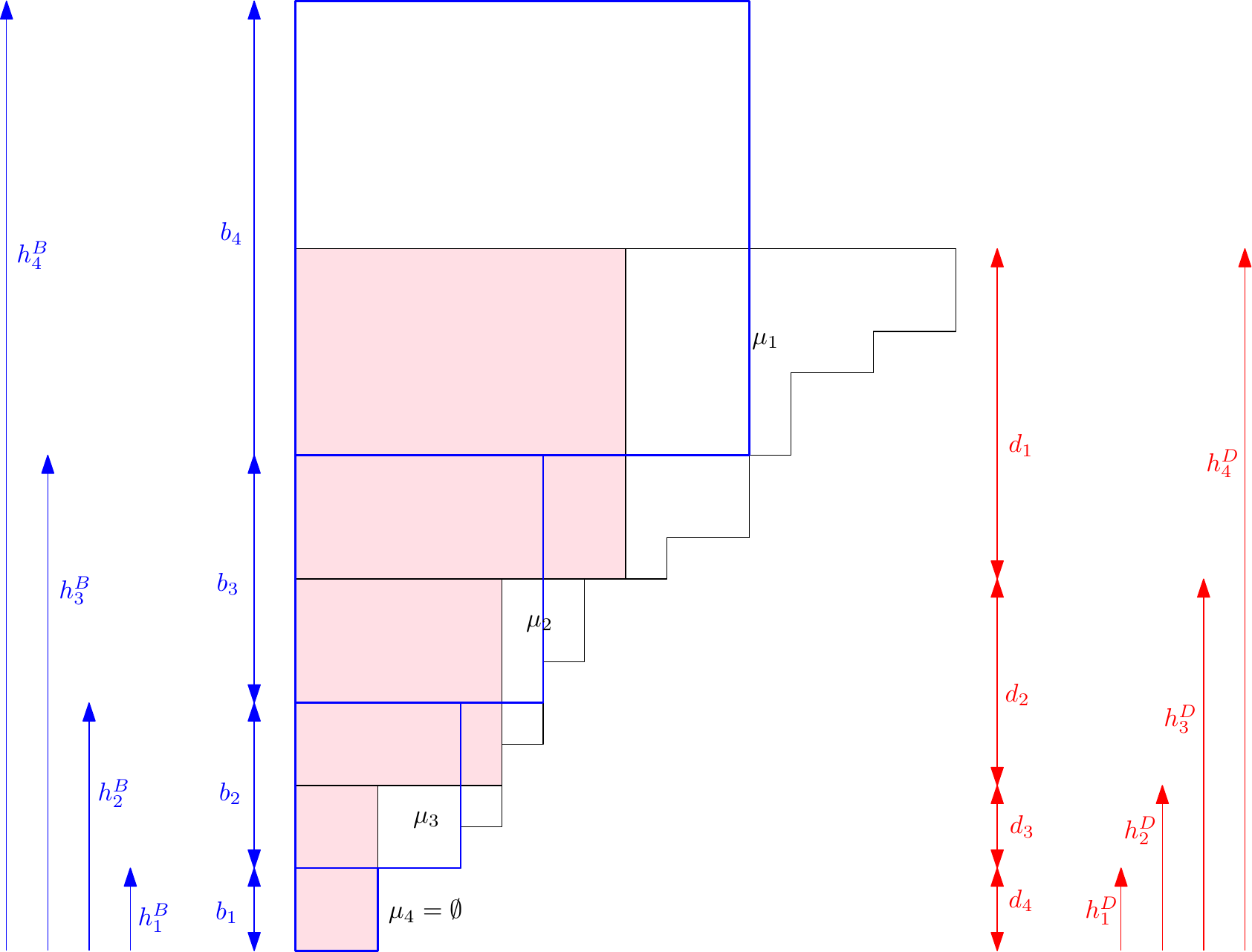}
\caption{Successive Durfee squares (in pink) and bottom squares (in blue) of a partition}
\label{fig:all_squares}
\end{figure}

Let us start with a proposition about the heights.

\begin{Proposition}
\label{prop:height_squares}
Let $\lambda$ be a partition having exactly $r-1$ non-empty Durfee squares. Let us consider the bottom dissection with only squares as well. Then for all $1 \leq j \leq r-1$, we have
$$h^D_j \leq h^B_j < h^D_{j+1}.$$
Moreover we have the equality $h^D_j = h^B_j$ if and only if the partitions $\mu_{r-1}, \dots , \mu_{r-j}$ have strictly less than $d_{r-1}, \dots, d_{r-j}$ parts, respectively.
\end{Proposition}
\begin{proof}
We proceed by induction on $j$.

\begin{itemize}
\item
For $j$=1, the first bottom square $B_1$ starts at the bottom of the partition.

If the box directly to the right of the bottom-right of the last Durfee square $D_{r-1}$ is empty, then the bottom square $B_1$ coincides with $D_{r-1}$ and we have $h^D_1 = h^B_1$. This is the case in the example of Figure \ref{fig:all_squares}.

If it is not empty, i.e. if the smallest part of the partition is larger than $d_{r-1}$, i.e. if $\mu_{r-1}$ has $d_{r-1}$ parts, then the first bottom square $B_1$ will be larger than the last Durfee square $D_{r-1}$, so we have $h^D_1 < h^B_1$. However, for our partition to be well-defined, the size of the last part cannot exceed the size of the Durfee square $D_{r-2}$, so we have $h^B_1= b_1 \leq d_{r-2}$, and therefore $h^B_1 < d_{r-1}+d_{r-2}= h^D_{2}.$

\item
Now assume that the proposition is true for all $k \leq j-1$, and prove it for $j$.
By the induction hypothesis, we have $h^D_{j-1} \leq h^B_{j-1} < h^D_{j},$ with
$h^D_{j-1} = h^B_{j-1}$ if and only if the partitions $\mu_{r-1}, \dots , \mu_{r-j+1}$ have strictly less than $d_{r-1}, \dots, d_{r-j+1}$ parts, respectively.

If $h^D_{j-1} = h^B_{j-1}$, then if the box directly to the right of the bottom-right of $D_{r-j+1}$ is empty (i.e. $\mu_{r-j}$ has less than $d_{r-j}$ parts), the bottom square $B_j$ will coincide again with the Durfee square $D_{r-j}$ and we will have $h^D_j = h^B_j$. Otherwise, if this box is not empty, then, as in the second case for $j=1$, the bottom square $B_j$ will be larger than the Durfee square $D_{r-j}$ and we will have $h^B_j= h^B_{j-1}+b_j < h^D_{j-1}+d_{r-j}= h^D_{j}.$

Now let us treat the case where $h^D_{j-1} < h^B_{j-1} < h^D_{j}.$ In general, by definition of the successive Durfee squares, we have $d_{r-j} \leq b_{j} \leq d_{r-j-1}$. Combining this with $h^D_{j-1} < h^B_{j-1} < h^D_{j}$ gives 
$h^D_{j-1}+d_{r-j}< h^B_{j-1}+b_{j} < h^D_{j}+ d_{r-j-1}$, i.e.
$h^D_j < h^B_j < h^D_{j+1}.$
\end{itemize}
\end{proof}

We can now easily prove the following theorem, which is the particular case $i=r$ of Theorem \ref{th:Durfee_Bottom}.

\begin{Theorem}
\label{th:all_squares}
For all $r \geq 2$, $\mathcal{B}_{r,r}=\mathcal{D}_{r,r}$.
\end{Theorem}
\begin{proof}
We prove that for all $r \geq 2$,
\begin{equation}
\label{eq:set_equality_squares}
\mathcal{D}_{r,r}\setminus \mathcal{D}_{r-1,r-1} = \mathcal{B}_{r,r}\setminus \mathcal{B}_{r-1,r-1}.
\end{equation}
The set $\mathcal{D}_{r,r}\setminus \mathcal{D}_{r-1,r-1}$ is the set of partitions having exactly $r-1$ successive non-empty Durfee squares, while $\mathcal{B}_{r,r}\setminus \mathcal{B}_{r-1,r-1}$ is the set of partitions such that in their bottom square dissection, the row above the $(r-1)$-th bottom square is empty, but not the one above the $(r-2)$-th bottom square.

Let $\lambda$ be a partition in $\mathcal{D}_{r,r}\setminus \mathcal{D}_{r-1,r-1}$. By Proposition \ref{prop:height_squares}, we have 
$$ h^B_{r-2} <h^D_{r-1} \leq h^B_{r-1} < h^D_{r} = +\infty.$$
The number of rows of the partition is given by $h^D_{r-1}$, so the row above $B_{r-1}$ is indeed empty, while the row above $B_{r-2}$ still belongs to the partition. Therefore $\lambda$ belongs to $\mathcal{B}_{r,r}\setminus \mathcal{B}_{r-1,r-1}$ as well.

We have proved that for all $r \geq 2$,
$\mathcal{D}_{r,r}\setminus \mathcal{D}_{r-1,r-1} \subseteq \mathcal{B}_{r,r}\setminus \mathcal{B}_{r-1,r-1}.$
But we also have that the set $\mathcal{P^*}$ of all non-empty partitions can be written as disjoint unions
$$\mathcal{P^*} =  \bigsqcup_{r \geq 2} \mathcal{D}_{r,r}\setminus \mathcal{D}_{r-1,r-1} = \bigsqcup_{r \geq 2} \mathcal{B}_{r,r}\setminus \mathcal{B}_{r-1,r-1},$$
so by Lemma \ref{lem:disjoint_union}, we actually have for all $r$, $\mathcal{D}_{r,r}\setminus \mathcal{D}_{r-1,r-1} = \mathcal{B}_{r,r}\setminus \mathcal{B}_{r-1,r-1}.$

Using the fact that $\mathcal{D}_{r,r}=(\mathcal{D}_{r,r}\setminus \mathcal{D}_{r-1,r-1}) \sqcup \mathcal{D}_{r-1,r-1} $, an immediate induction on $r$ proves the desired result.
\end{proof}

\subsection{The case with only rectangles: $\mathcal{B}_{r,1}=\mathcal{D}_{r,1}$}
We now turn to the situation in which only rectangles appear (case $i=1$), which will also be useful in the proof of the general case. It is very similar to the case of only squares, so we will only sketch the proofs which work in the exact same way.

In a partition $\lambda$ having exactly $r-1$ non-empty Durfee rectangles, we name again $\mu_1, \dots, \mu_{r-1}$ the partitions to the right of the Durfee rectangles $D'_1, \dots , D'_{r-1}.$ As before, we start with a proposition about the heights.

\begin{Proposition}
\label{prop:height_rectangles}
Let $\lambda$ be a partition having exactly $r-1$ non-empty Durfee rectangles (and no squares). Let us consider the bottom dissection with only rectangles. Then for all $1 \leq j \leq r-1$, we have
$$h^D_j \leq h^B_j < h^D_{j+1}.$$
Moreover we have the equality $h^D_j = h^B_j$ if and only if the partitions $\mu_{r-1}, \dots , \mu_{r-j}$ have strictly less than $d'_{r-1}, \dots, d'_{r-j}$ parts, respectively.
\end{Proposition}
\begin{proof}
The proof is exactly the same as the proof of Proposition \ref{prop:height_squares}, but with ``squares'' replaced by ``rectangles'' and $d_k$ (resp. $b_k$) replaced by $d'_k$ (resp. $b'_k$) for all $k$.
\end{proof}

As in the previous section, we use this proposition to prove the following theorem, which is the particular case $i=1$ of Theorem \ref{th:Durfee_Bottom}.

\begin{Theorem}
\label{th:all_rectangles}
For all $r \geq 2$, $\mathcal{B}_{r,1}=\mathcal{D}_{r,1}$.
\end{Theorem}
\begin{proof}
We prove that for all $r \geq 2$,
\begin{equation}
\label{eq:set_equality_rectangles}
\mathcal{D}_{r,1}\setminus \mathcal{D}_{r-1,1} = \mathcal{B}_{r,1}\setminus \mathcal{B}_{r-1,1}.
\end{equation}
The set $\mathcal{D}_{r,1}\setminus \mathcal{D}_{r-1,1}$ is the set of partitions having exactly $r-1$ successive non-empty Durfee rectangles, while $\mathcal{B}_{r,1}\setminus \mathcal{B}_{r-1,1}$ is the set of partitions such that in their bottom rectangle dissection, the row above the $(r-1)$-th bottom rectangle is empty, but not the one above the $(r-2)$-th bottom rectangle.

Let $\lambda$ be a partition in $\mathcal{D}_{r,1}\setminus \mathcal{D}_{r-1,1}$. By Proposition \ref{prop:height_rectangles}, we have 
$$ h^B_{r-2} <h^D_{r-1} \leq h^B_{r-1} < h^D_{r} = +\infty.$$
The number of rows of the partition is given by $h^D_{r-1}$, so the row above $B'_{r-1}$ is indeed empty, while the row above $B'_{r-2}$ still belongs to the partition. Therefore $\lambda$ belongs to $\mathcal{B}_{r,1}\setminus \mathcal{B}_{r-1,1}$ as well.

We have proved that for all $r \geq 2$,
$\mathcal{D}_{r,r}\setminus \mathcal{D}_{r-1,r-1} \subseteq \mathcal{B}_{r,r}\setminus \mathcal{B}_{r-1,r-1}.$ We now use Lemma \ref{lem:disjoint_union} again to prove the reverse inclusion.

Here there is a slight subtlety, as a partition has a well-defined Durfee (resp. bottom) dissection with only rectangles if and only if the smallest part of the partition is $>1$. Indeed, the only Durfee (resp. bottom) rectangle that a part $1$ fits in is a rectangle of length $1$ and height $0$, which is allowed in our definitions, but never permits to move up in the partition unless squares are allowed (a square of size $1 \times 1$ increases the height by $1$ and allows us to move up).

Thus we have the disjoint unions
$$\mathcal{P}^*_1 =  \bigsqcup_{r \geq 2} \mathcal{D}_{r,1}\setminus \mathcal{D}_{r-1,1} = \bigsqcup_{r \geq 2} \mathcal{B}_{r,1}\setminus \mathcal{B}_{r-1,1},$$
where $\mathcal{P}^*_1$ is the set of all partitions with parts $>1$ .
By Lemma \ref{lem:disjoint_union}, we actually have for all $r$, $\mathcal{D}_{r,1}\setminus \mathcal{D}_{r-1,1} = \mathcal{B}_{r,r}\setminus \mathcal{B}_{r-1,1}.$

Using the fact that $\mathcal{D}_{r,1}=(\mathcal{D}_{r,1}\setminus \mathcal{D}_{r-1,1}) \sqcup \mathcal{D}_{r-1,1} $, an immediate induction on $r$ proves the desired result.
\end{proof}

\subsection{The general case}
In this subsection, we use our previous results on squares and rectangles to prove the general case of Theorem \ref{th:Durfee_Bottom}.

Our strategy is the same as in the proofs of Theorems \ref{th:all_squares} and \ref{th:all_rectangles}, but there are some more technicalities, so we cut the proof in two steps. We start with the following.

\begin{Proposition}
\label{prop:general_case}
Let $r \geq 2$ and $1 < i < r$ be two integers. We have
$$\mathcal{D}_{r,i}\setminus \mathcal{D}_{r-1,i-1} \subseteq \mathcal{B}_{r,i}\setminus \mathcal{B}_{r-1,i-1}.$$
\end{Proposition}
\begin{proof}
First, describe the two sets we are considering.
The set $\mathcal{D}_{r,i}\setminus \mathcal{D}_{r-1,i-1}$ is the set of partitions with exactly $r-i$ non-empty Durfee rectangles followed by exactly $i-1$ non-empty Durfee squares. The partition of Figure \ref{fig:bottom_durfee_squares} shows a partition in $\mathcal{D}_{5,3}\setminus \mathcal{D}_{4,2}$.

On the other hand, $\mathcal{B}_{r,i}\setminus \mathcal{B}_{r-1,i-1}$ is the set of partitions such that in their $(i-1)$-bottom dissection, the row above $B'_{r-1}$ is empty, but such that in their $(i-2)$-bottom dissection, the row above $B'_{r-2}$ is not empty. In other words, if we draw $i-1$ bottom squares followed by $r-i$ bottom rectangles, we exit the partition, but not if we draw $i-2$ bottom squares followed by $r-i$ bottom rectangles.
The reader can check that the partition of Figure \ref{fig:bottom_durfee_squares} is also in $\mathcal{B}_{5,3}\setminus \mathcal{B}_{4,2}$.

\medskip
Let us start with a very particular case: the partition $\lambda^1:=1^{i-1}$. This partition belongs to $\mathcal{D}_{r,i}\setminus \mathcal{D}_{r-1,i-1}$ for all $r$, as the $(r-i)$-Durfee dissection consists of $r-i$ Durfee rectangles of size $1 \times 0$ followed by $i-1$ Durfee squares of size $1$.

Now we check that $\lambda^1$ also belongs to $\mathcal{B}_{r,i}\setminus \mathcal{B}_{r-1,i-1}.$ In the $(i-1)$-bottom dissection of $\lambda^1$, we start with $i-1$ bottom squares of size $1$, and the row above the last square is empty. So in particular for all $r$, the part above $B'_{r-1}$ is empty, so $\lambda^1 \in \mathcal{B}_{r,i}.$ However, for all $r$ if we draw the $(i-2)$-bottom dissection of $\lambda^1$, we will first have $(i-2)$ Durfee squares of size $1$, but then we will have infinitely many bottom rectangles of size $1 \times 0$ and never exit the partition, so $\lambda^1 \notin \mathcal{B}_{r-1,i-1}.$

More generally, any partition starting with $1^{i-1}$ belongs to $\mathcal{D}_{r,i}\setminus \mathcal{D}_{r-1,i-1}$ if and only if its restriction to its parts different from $1$ belongs to $\mathcal{D}_{r-i,1}$. In the same way, any partition starting with $1^{i-1}$ belongs to $\mathcal{B}_{r,i}\setminus \mathcal{B}_{r-1,i-1}$ if and only if its restriction to its parts different from $1$ belongs to $\mathcal{B}_{r-i,1}$. But we already proved in Theorem \ref{th:all_rectangles} that $\mathcal{D}_{r-i,1}=\mathcal{B}_{r-i,1}$. Thus any partition starting with $1^{i-1}$ belongs to $\mathcal{D}_{r,i}\setminus \mathcal{D}_{r-1,i-1}$ if and only if it belongs to $\mathcal{B}_{r,i}\setminus \mathcal{B}_{r-1,i-1}.$

\medskip
Let us now turn to the general case. Start by noticing that for any partition, the $(r-2)$-th bottom rectangle in the $(i-2)$-bottom dissection always ends lower than the $(r-2)$-th bottom square/rectangle in the $(i-1)$-bottom dissection. Indeed, the first $i-2$ bottom squares coincide in both dissections, but the $(i-1)$-th is a rectangle in the first case while it is a square in the second one. So the height of $B'_{i-1}$ in the $(i-2)$-bottom dissection is exactly one less than the height of $B_{i-1}$ in the $(i-1)$-bottom dissection. Then from this point on, all bottom rectangles $B'_k$ will remain lower in the $(i-2)$-bottom dissection than in the $(i-1)$-bottom dissection.

Figure \ref{fig:general1} shows the successive Durfee squares of a partition, together with its successive bottom squares, and the corresponding heights.

\begin{figure}[H]
\includegraphics[width=0.65\textwidth]{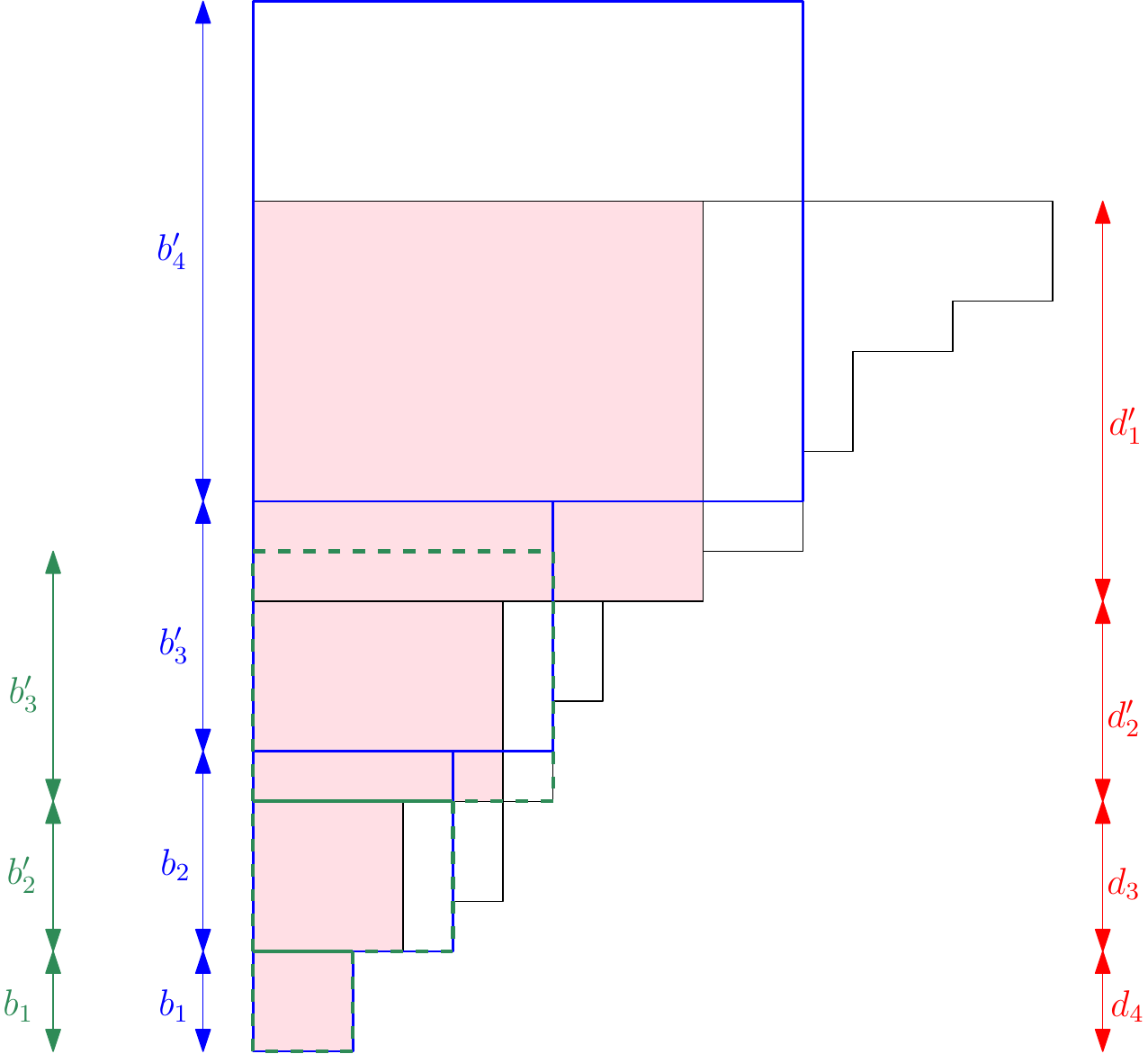}
\caption{$2$-Durfee dissection (in pink), $2$-bottom dissection (in blue), and $1$-bottom dissection (in green) of a partition}
\label{fig:general1}
\end{figure}
 
Thus if in a partition $\lambda$, the last row of the $(r-2)$-th bottom square/rectangle in the $(i-1)$-bottom dissection still belongs to the partition, i.e.
$$h^B_{r-2} \leq h^D_{r-1},$$
then $\lambda$ is clearly not in $\mathcal{B}_{r-1,i-1}.$
On the other hand, $\lambda \in \mathcal{B}_{r,i}$ if and only if
$$h^D_{r-1} \leq h^B_{r-1}.$$
We will show that for all partitions of $\mathcal{D}_{r,i}\setminus \mathcal{D}_{r-1,i-1},$
\begin{equation}
\label{eq:toprove}
h^B_{r-2} \leq h^D_{r-1}\leq h^B_{r-1}.
\end{equation}

\medskip

Take a partition $\lambda$ in $\mathcal{D}_{r,i}\setminus \mathcal{D}_{r-1,i-1}$ together with its $(r-i)$-Durfee dissection and draw its $(i-1)$-bottom dissection starting with $i-2$ bottom squares. By Proposition \ref{prop:height_squares}, we have
\begin{equation}
\label{eq:lastsquare}
h^D_{i-2} \leq h^B_{i-2} < h^D_{i-1}.
\end{equation}

Moreover, $D_{r-i-1}$ is the highest Durfee square, so above it is a Durfee rectangle $D'_{r-i}$. By \eqref{eq:lastsquare} and the definition of Durfee squares/rectangles, we have
\begin{equation}
\label{eq:transition}
d_{r-i+1} \leq b_{i-1} \leq d'_{r-i}+1.
\end{equation}

Thus we can add \eqref{eq:lastsquare} and \eqref{eq:transition} and obtain that
$$h^D_{i-1} \leq h^B_{i-1} \leq h^D_{i-1}.$$
From then on, we only have bottom and Durfee rectangles, so the induction step of Proposition \ref{prop:height_rectangles} can be applied to show that for all  $i-1 \leq j \leq r-1$, we have
$$h^D_j \leq h^B_j < h^D_{j+1}.$$
Thus \eqref{eq:toprove} is proved and we have shown that $\lambda \in \mathcal{B}_{r,i}\setminus \mathcal{B}_{r-1,i-1}.$

\end{proof}

We can now use Lemma \ref{lem:disjoint_union} again to finish the proof of Theorem \ref{th:Durfee_Bottom}.
\begin{proof}[Proof of Theorem \ref{th:Durfee_Bottom}]
Let $r \geq 2$ and let $1 < k <r$. Denoting by $\mathcal{P}$ the set of all partitions, we have the disjoint unions
\begin{align*}
\mathcal{P} &= \mathcal{D}_{k,1} \sqcup \bigsqcup_{j \geq 1} (\mathcal{D}_{k+j,j+1}\setminus \mathcal{D}_{k+j-1,j})\\
&= \mathcal{B}_{k,1} \sqcup \bigsqcup_{j \geq 1} (\mathcal{B}_{k+j,j+1}\setminus \mathcal{B}_{k+j-1,j}).
\end{align*}
We already know from Theorem \ref{th:all_rectangles} that $\mathcal{D}_{k,1}=\mathcal{B}_{k,1}$. Moreover, from Proposition \ref{prop:general_case} with $r=k+j$, $i=j+1$, we know that $(\mathcal{D}_{k+j,j+1}\setminus \mathcal{D}_{k+j-1,j}) \subseteq (\mathcal{B}_{k+j,j+1}\setminus \mathcal{B}_{k+j-1,j})$ for all $k$ and $j$. This completes the proof.
\end{proof}

\section{A commutative algebra proof of the equality $\mathcal{C}_{r,i}=\mathcal{D}_{r,i}$}
\label{sec:alg_Durf_Bot}

In this section, we consider the graded polynomial ring $\textbf{K}[x_1,x_2,\ldots]$; the grading is induced by the weights of the variables, $x_i$ being of weight $i$.
Given an integer $r\geq 2,$ for $1\leq i \leq r,$ recall from the introduction the ideal $I_{r,i}\subset \textbf{K}[x_1,x_2,\ldots]$ generated by $x_1^i$ and the monomials of the following form:

$$\underbrace{x_{n_{1,1}}}_{\text {first block}} \underbrace{ x_{n_{2,1}}\cdots x_{n_{2,f_{r,i}(2)}}}_{\text {second block}} \underbrace{ x_{n_{3,1}} \cdots x_{n_{3,f_{r,i}(3)}}}_{\text{third block}} \cdots \underbrace{ x_{n_{r,1}}\cdots x_{n_{r,f_{r,i}(r)}}}_{\text{$r$-th block}}\,,$$

\noindent where 

$$f_{r,i}(j):= \begin{cases}
1 &\text{ if } j=1, \\
n_{j-1,f_{r,i}(j-1)} &\text{ if } 2\leq j \leq i,\\
n_{j-1,f_{r,i}(j-1)}-1 &\text{ if } i+1 \leq j \leq r.
\end{cases}$$

One can show by induction that for all $1\leq i,j \leq r-1$,
$$f_{r,i}(j)=f_{r-1,i}(j)\  \text{and} \ f_{r,r}(j)=f_{r-1,r-1}(j).$$
From now on, for simplicity, we use $f(j)$ instead of $f_{r,i}(j)$ when there is no confusion (namely when $r$ and $i$ are fixed).\\

The partitions associated with the monomials of $\textbf{K}[x_1,x_2,\cdots]/I_{r,i}$ are (almost by definition) exactly the partitions of $\mathcal{C}_{r,i}$ (see also Proposition~5.3 in \cite{A}): indeed, these monomials are not divisible by any of the monomials generating the ideal $I_{r,i}.$ Note that for the partition $\lambda$ associated with the monomial   
$$x_\lambda=x_{n_{1,1}}  x_{n_{2,1}}\cdots x_{n_{2,f(2)}}x_{n_{3,1}} \cdots x_{n_{3,f(3)}} \cdots  x_{n_{r,1}}\cdots x_{n_{r,f(r)}},$$
we have $p_{i,\ell}(\lambda)=n_{\ell,f(\ell)}$ for all $\ell=1,\ldots,r-1$. Actually, this is what motivates the introduction of the 
 $p_{i,\ell}(\lambda)$'s and the condition that the length $s$ of $\lambda$ satisfies 
 $$s= \sum_{j=1}^{r-1} p_{i,j}(\lambda)-(r-i)+1.$$ 
 Moreover, for a partition $\mu$ of length $s'$ associated with a multiple of the monomial $x_\lambda$,  we have $p_{i,j}(\mu)\leq p_{i,j}(\lambda)$ and
 $$ s'\geq s>\sum_{j=1}^{r-1} p_{i,j}(\lambda)-(r-i)\geq \sum_{j=1}^{r-1} p_{i,j}(\mu)-(r-i).$$
 This shows that the partitions associated with a monomial in $I_{r,i}$ are exactly those which are not in $\mathcal{C}_{r,i}$.\\

 The goal of this section is to prove that the partitions associated with the monomials of $\textbf{K}[x_1,x_2,\cdots]/I_{r,i}$ are exactly the partitions of $\mathcal{D}_{r,i}$, i.e. $\mathcal{C}_{r,i}=\mathcal{D}_{r,i}.$  To do so, we analyse the monomials which do not belong to $I_{r,i}$. The proof is by induction on $r$ and is guided by the following exact sequence: for $i =1,\ldots,r-1$, we have 

\begin{equation}\label{seq}
0\longrightarrow \frac{I_{r-1,i}}{I_{r,i}} \longrightarrow \frac{\textbf{K}[x_1,x_2\cdots]}{I_{r,i}} \longrightarrow \frac{\textbf{K}[x_1,x_2,\cdots]}{I_{r-1,i}} \longrightarrow 0,
\end{equation}
where we have $I_{r,i}\subset I_{r-1,i}$. Thus a monomial basis of $\textbf{K}[x_1,x_2,\cdots]/I_{r,i}$ is formed by a monomial basis of $\textbf{K}[x_1,x_2,\cdots]/I_{r-1,i}$ and a monomial basis of $I_{r-1,i}/I_{r,i}$. By induction, a monomial basis of $\textbf{K}[x_1,x_2,\cdots]/I_{r-1,i}$ is formed by the monomials associated with partitions in $\mathcal{D}_{r-1,i}$. A monomial basis of $I_{r-1,i}/I_{r,i}$ is formed by monomials of the form 

\begin{equation}\label{monome}
x_\lambda =x_{n_{1,1}}x_{n_{2,1}}\cdots x_{n_{2,f(2)}} x_{n_{3,1}} \cdots x_{n_{3,f(3)}} \cdots x_{n_{r-1,1}} \cdots x_{n_{r-1,f(r-1)}}  x_{n_{r,1}}\cdots x_{n_{r,\ell}},
\end{equation}
where $0\leq \ell < f(r)$. Note that when $\ell=0$, the last variable of $x_{\lambda}$ is $x_{n_{r-1,f(r-1)}}$. We start by stating the main theorem of this section, which provides an algebraic proof of Theorem~\ref{th:Durfee_Bottom}.

\begin{Theorem}\label{thm:algebraic}
Let $r\geq 2$ and $1\leq i \leq r$ be integers, let $\lambda$ be a partition and $x_{\lambda}$ be the associated monomial. We have $x_\lambda   \in \textbf{K}[x_1,x_2,\cdots]/I_{r,i}$ if and only if $\lambda   \in \mathcal{D}_{r,i}.$ Equivalently, $\mathcal{C}_{r,i}=\mathcal{D}_{r,i}.$
\end{Theorem}
As several preliminary results are necessary, we postpone the proof of this theorem to the end of this section, and start by presenting the results which will be key in the proof.

The analysis of the partitions associated with the monomials of the ring $\textbf{K}[x_1,x_2,\cdots]/I_{r-1,i}$ at the right of the exact sequence~(\ref{seq}) is understood by induction. 
We now focus on the monomials of the module $I_{r-1,i}/I_{r,i}$, which are of the form~(\ref{monome}). 

\begin{Lemma}\label{fst. rec. } Let $$x_\lambda =x_{n_{1,1}}x_{n_{2,1}}\cdots x_{n_{2,f(2)}} x_{n_{3,1}} \cdots x_{n_{3,f(3)}} \cdots  x_{n_{r,1}}\cdots x_{n_{r,\ell}},$$ 
where $0\leq \ell < f(r).$ If  $1\leq i \leq r-1$ (respectively $i=r$) then the first Durfee rectangle (respectively the first Durfee square) ends at the $(r-1)$-th block of the monomial $x_\lambda.$
\end{Lemma}
\begin{proof} Let $1\leq i \leq r$ (resp. $i=r$) and consider the  partition $\lambda$ associate to $x_{\lambda}$. Denote the height of the first Durfee rectangle of $\lambda$  by $d'_1$ (resp. the size of its first Durfee square by $d_1$). We prove that if we remove the monomial associated with the partition whose parts are the first $d'_1$  (resp. $d_1$)  parts of $\lambda$ from $x_\lambda$ then we obtain
$$x_{n_{1,1}}x_{n_{2,1}}\cdots x_{n_{2,f(2)}} x_{n_{3,1}} \cdots x_{n_{3,f(3)}} \cdots  x_{n_{r-1,1}}\cdots x_{n_{r-1,\ell'}}$$
where $0\leq \ell' < f(r-1)$.
To do so, we need to prove that if $1\leq i \leq r-1$ (resp. $i=r$), the height $d'_1$ of the first Durfee rectangle (resp. the size $d_1$ of the first Durfee square) satisfies:
\begin{equation}
\label{eq:star}
\ell < d'_1 \text{ (resp. $d_1$)} \leq f(r-1)+\ell .
\end{equation}
\begin{itemize}

\item[-] Left inequality of \eqref{eq:star}:
for $1\leq i \leq r-1$, the first inequality follows from the fact that $\ell<f(r)=n_{r-1,f(r-1)}-1$. We then have $\ell+1\leq n_{r-1,f(r-1)}-1$, and $\lambda$ contains a rectangle of height $\ell+1$, which is not necessarily of maximal size.
For $i=r,$ it follows from the fact that $\ell<f(r)=n_{r-1,f(r-1)}$. Thus $\ell+1\leq n_{r-1,f(r-1)}$, and in this case $\lambda$ has a Durfee square of size at least $\ell+1$.
\item[-] Right inequality of \eqref{eq:star}:
note that 
$$f(r-1)= \begin{cases}
n_{r-2,f(r-2)} &\text{ if } r-1\leq i \leq r, \\
n_{r-2,f(r-2)}-1 &\text{ if } 1\leq i<r-1.
\end{cases}$$
So for $1\leq i \leq r-1$,
$$\ell+f(r-1)+1\geq \ell+(n_{r-2,f(r-2)}-1)+1\geq n_{r-2,f(r-2)}.$$
This proves that if $1\leq i \leq r-1$, then $\lambda$ cannot have a Durfee rectangle of height larger than or equal to $\ell+f(r-1)+1$.
\\
For $i=r$,
$$\ell+f(r-1)+1= \ell+n_{r-2,f(r-2)}+1 > n_{r-2,f(r-2)}.$$
So in this case, the size of the first Durfee square of $\lambda$ is at most $\ell+f(r-1)$.
\end{itemize}
\end{proof}
\begin{Proposition}\label{1} If 
$$x_\lambda =x_{n_{1,1}}x_{n_{2,1}}\cdots x_{n_{2,f(2)}} x_{n_{3,1}} \cdots x_{n_{3,f(3)}} \cdots  x_{n_{r,1}}\cdots x_{n_{r,\ell}},$$ 
where $0\leq \ell < f(r),$ then $\lambda$ belongs to $\mathcal{D}_{r,i} $ and has exactly $r-i$  Durfee rectangles of height $>0$ followed by $i-1$ Durfee squares.
\end{Proposition}

\begin{proof}
The proof is by induction on $r$. For $r=2$ and $i=1$ (resp. $i=2$), suppose that $x_{\lambda}=x_{n_{1,1}}x_{n_{2,1}} \cdots x_{n_{2,\ell}}$ with $0\leq \ell <f_{2,1}(2)$ (resp. $0 \leq \ell<f_{2,2}(2)$). Since $0 \leq \ell  <f_{2,1}(2)=n_{1,1}-1$  (resp. $0 \leq \ell<f_{2,2}(2)=n_{1,1}$), $\lambda$ has a  Durfee rectangle of height (resp. a Durfee square of size) $\ell+1$ and no part below it (see Figure \ref{fig:prop43}).

\begin{figure}[H]
\includegraphics[width=0.5\textwidth]{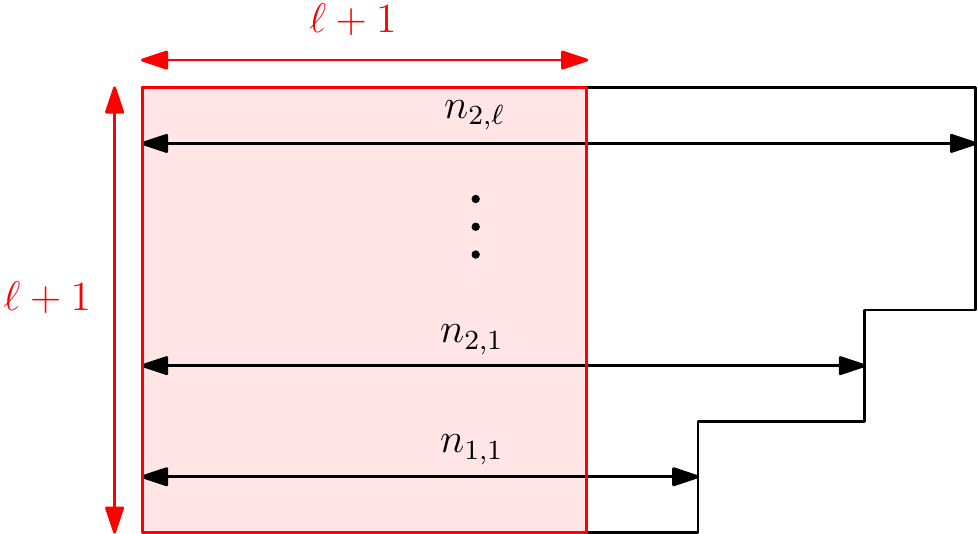}
\caption{The partition $\lambda$ associated to $x_{n_{1,1}}x_{n_{2,1}}\dots x_{n_{2,l}}$ with $0 \leq \ell < n_{1,1}$}
\label{fig:prop43}
\end{figure}

Let us now assume  that the proposition is true for $r-1$, and prove it for $r$.
Suppose that 
$$x_\lambda =x_{n_{1,1}}x_{n_{2,1}}\cdots x_{n_{2,f(2)}} x_{n_{3,1}} \cdots x_{n_{3,f(3)}} \cdots  x_{n_{r,1}}\cdots x_{n_{r,\ell}},$$ 
where $0\leq \ell < f(r)$. Lemma \ref{fst. rec. } says that by removing from $x_\lambda$ the monomial associated with the partition $\nu$ whose parts are the first parts of $\lambda$ containing its first square/rectangle,  we obtain the following monomial:

$$x_{\mu}:=x_{n_{1,1}}x_{n_{2,1}}\cdots x_{n_{2,f(2)}} x_{n_{3,1}} \cdots x_{n_{3,f(3)}} \cdots  x_{n_{r-1,1}}\cdots x_{n_{r-1,\ell'}}\,,$$
where $0\leq \ell' < f(r-1)$.



Now, by the induction hypothesis applied to $x_{\mu}$,
we obtain that
\begin{itemize}
\item[-]for $1\leq i \leq r-1$, the partition $\mu$ associated to $x_{\mu}$ has exactly $r-1-i$ Durfee rectangles and $i-1$ Durfee squares;
\item[-]for $i=r$, $f(j)=f_{r,r}(j)=f_{r-1,r-1}(j)$ for all $1\leq j \leq r-1$, so by the induction hypothesis $\mu$ has exactly $r-2$ Durfee squares.  
\end{itemize}
 
Adding the partition $\nu$ (which was removed earlier from $\lambda$ and which contains its first Durfee square/rectangle) to $\mu$ proves that $\lambda$ has exactly $r-i$ Durfee rectangles and $i-1$ Durfee squares. 
 \end{proof}

Next we need to fix some notation. Set $1\leq s \leq i-1$ and $\lambda \in \mathcal{D}_{r,i}$ with exactly $r-i$ Durfee rectangles of heights $d'_1\geq \cdots \geq d'_{r-i}>0$ (or with no Durfee rectangle if $i=r$) and $s$ Durfee squares of sizes $d_{r-i+1}\geq \cdots \geq d_{r-i+s}$. We define $S_{\lambda}$ to be the set of all  $1\leq j \leq s$ such that the first part of the $j$-th square is strictly less than the size of the $(j-1)$-th square (or (the height of the $(r-i)$-th rectangle$+1$)  if $j=1$). In other words,
$$ S_{\lambda}=\{1\leq j \leq s\,|\, \lambda_{\sum_{l=1}^{r-i}d'_l+\sum_{l=1}^{j-1}d_{r-i+l}+1}< d_{r-i+j-1} \ \text{(or $\lambda_{\sum_{l=1}^{r-i}d'_l+1}<d'_{r-i}+1$ if $j=1$ and $1\leq i \leq r-1$}) \}.$$
If $S_{\lambda} \neq \emptyset$, then we define $m_{\lambda}:= min(S_{\lambda})$. See an example on Figure~\ref{fig:lem44}.


\begin{figure}[H]
\includegraphics[width=0.5\textwidth]{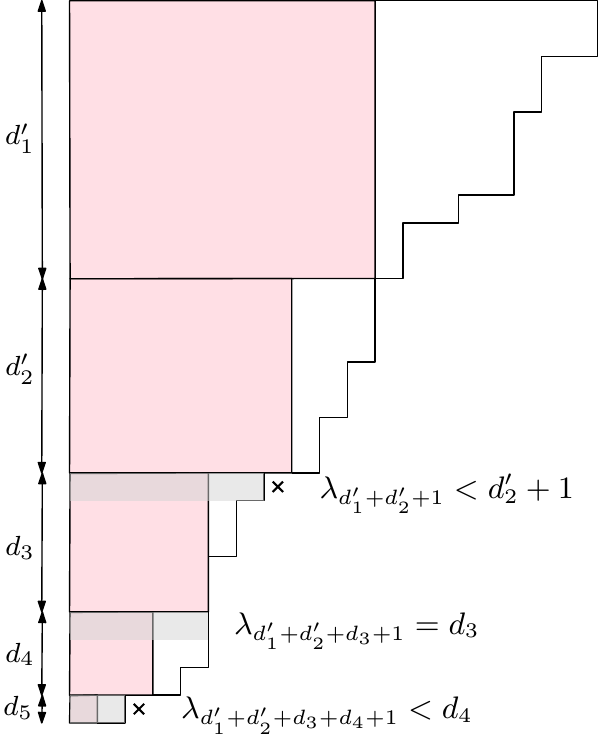}
\caption{A partition $\lambda \in \mathcal{D}_{7,5}$ with $S_{\lambda}=\{1,3\}$ and $m_{\lambda}=1$, for $s=3$}
\label{fig:lem44}
\end{figure}

\begin{Lemma}\label{43} Let $r\geq 2$. We set $\mathcal{D}_{r,i}:=\mathcal{D}_{r,r}$ for all $i >r$ and $\mathcal{D}_{1,1}:= \emptyset$. We have:
 \begin{itemize}
 \item $\mathcal{D}_{r,1} \setminus \mathcal{D}_{r-1,1}$ is the set of all partitions $\lambda$ with exactly $r-1$ Durfee rectangles of heights $>0$ and nothing after their last rectangle.
 \item $\mathcal{D}_{r,r} \setminus \mathcal{D}_{r-1,r}=\mathcal{D}_{r,r} \setminus \mathcal{D}_{r-1,r-1}$  is the set of all partitions $\lambda$ with exactly $r-1$ Durfee squares.
 \item If $2\leq i \leq r-1$, then $\mathcal{D}_{r,i} \setminus \mathcal{D}_{r-1,i}$ is the set of all partitions $\lambda$ with exactly $r-i$ Durfee rectangles of heights $>0$ and $i-1$ Durfee squares with one of the following conditions:
 \begin{itemize}
 \item $S_{\lambda}= \emptyset$ and $d_{r-1}>1;$
 \item $S_{\lambda} \neq \emptyset$ and $m_{\lambda}=1;$
 \item $S_{\lambda} \neq \emptyset, 2\leq m_{\lambda} \leq i-1$, and $d_{r-i+m_\lambda-1}>1$.
 \end{itemize}
 \end{itemize}
 \end{Lemma}

\begin{proof}
 The proof of the first two items is obvious by definition of $\mathcal{D}_{r,i}.$
 In order to prove the third one, let us consider a partition $\lambda \in \mathcal{D}_{r,i}$ for $2\leq i \leq r-1$. 
 
\begin{itemize}
\item If $\lambda$ has at most $\ell$ Durfee rectangles with $1\leq \ell \leq r-i-1$, or if it has $r-i-1$ Durfee rectangles and $i-1$ Durfee squares with at least one Durfee rectangle of height zero, then by definition $\lambda \in \mathcal{D}_{r-1,i}$.
 \item If $\lambda$ has exactly $r-i$ Durfee rectangles of heights $d'_1\geq \cdots \geq d'_{r-i}>0$ and nothing after its last rectangle, then it also has $r-i-1$ Durfee rectangles of heights $d'_1 \geq \cdots \geq d'_{r-i-1}$ and one Durfee square of size $d'_{r-i}$, and thus  $\lambda \in \mathcal{D}_{r-1,i}$. 
 \item If $\lambda$ has exactly $r-i$ Durfee rectangles of heights $d'_1\geq \cdots \geq d'_{r-i}>0$ and $s$ Durfee squares of sizes $d_{r-i+1} \geq \cdots \geq d_{r-i+s}$ with $1\leq s \leq i-1$:
 \begin{itemize}
 \item If $S_\lambda = \emptyset$, then in order to  determine the $(r-i-1)-$Durfee  dissection of $\lambda,$ we consider the first $r-i-1$ rectangles (obtained from its $(r-i)$-Durfee dissection) and we extend the $(r-i)$-th Durfee rectangle (again of its $(r-i)$-Durfee dissection) to the bottom in order to obtain a Durfee square of size $d'_{r-i}+1$. We then obtain  the $(r-1-i)$-Durfee dissection by considering the $s-1$  Durfee squares which are induced by the first $s-1$  Durfee squares of the $(r-i)$-Durfee dissection as follows:
   
\begin{itemize} 
\item If $d_{r-i+s}=1$, then $\lambda$ has exactly $r-i-1$ Durfee rectangles of heights $d'_1 \geq \cdots \geq  d'_{r-i-1}$ and $s$ Durfee squares of sizes $(d'_{r-i}+1)\geq d_{r-i+1} \geq \cdots \geq d_{r-i+s-1}$, and nothing below. Thus $\lambda \in \mathcal{D}_{r-1,i}$. 
\item If $d_{r-i+s}>1$, then $\lambda$  has exactly $r-i-1$ Durfee rectangles of heights $d'_1 \geq \cdots \geq d'_{r-i-1}$ and $s+1$ Durfee squares of sizes $(d'_{r-i}+1)\geq d_{r-i+1} \geq \cdots \geq d_{r-i+s-1} \geq (d_{r-i+s}-1)$. Thus if $1\leq s \leq i-2$, then $\lambda  \in \mathcal{D}_{r-1,i}$ and if  $s=i-1$, then $\lambda \not \in \mathcal{D}_{r-1,i}$. 
\end{itemize}
 \item If $S_\lambda \neq \emptyset$ and $m_\lambda=1$, then in order to determine the Durfee square right after the $(r-i-1)$-th Durfee rectangle, we have to reduce the $(r-i)$-th Durfee rectangle from the right to obtain a Durfee square of size $d'_{r-i}$. Then $\lambda$ has exactly $r-i-1$ Durfee rectangles of heights $d'_1\geq \cdots \geq d'_{r-i-1}>0$ and $s+1$ Durfee squares of sizes $d'_{r-i}\geq d_{r-i+1} \geq \cdots \geq d_{r-i+s}$. Thus if $1\leq s \leq i-2$, then $\lambda  \in \mathcal{D}_{r-1,i}$ and if  $s=i-1$, then $\lambda \not \in \mathcal{D}_{r-1,i}$. 
  \item If $S_\lambda \neq \emptyset$ and $2\leq m_\lambda\leq s$,  then in order to determine the Durfee square right after the $(r-i-1)$-th Durfee rectangle, we extend the $(r-i)$-th Durfee rectangle to the bottom in order to obtain a Durfee square of size $d'_{r-i}+1$. 
We then obtain in the $(r-1-i)$-Durfee dissection the first $m_{\lambda}-2$  Durfee squares which are induced by the first $m_{\lambda}-2$ Durfee squares appearing in the $(r-i)$-Durfee dissection as follows:  

\begin{itemize}
 \item If $d_{r-i+m_{\lambda}-1}=1$, then $\lambda$ has exactly $r-i-1$ rectangles of heights $d'_1 \geq \cdots \geq d'_{r-i-1}>0$ and $s$ squares of sizes $(d'_{r-i}+1)\geq d_{r-i+1} \geq \cdots \geq d_{r-i+m_{\lambda}-2}\geq d_{r-i+m_{\lambda}}\geq \cdots \geq d_{r-i+s}$, and therefore $\lambda \in \mathcal{D}_{r-1,i}$. Note that in this case $d_{r-i+m_\lambda}=\cdots =d_{r-i+s}=1$.
 \item If $d_{r-i+m_{\lambda}-1}>1$, then $\lambda$ has exactly $r-i-1$ Durfee rectangles of heights $d'_1 \geq \cdots \geq d'_{r-i-1}>0$ and $s+1$ Durfee squares of sizes $(d'_{r-i}+1)\geq d_{r-i+1} \geq \cdots \geq d_{r-i+m_{\lambda}-2}\geq (d_{r-i+m_{\lambda}-1}-1) \geq d_{r-i+m_{\lambda}}\geq \cdots \geq d_{r-i+s}$. Thus if $1\leq s \leq i-2$, then $\lambda  \in \mathcal{D}_{r-1,i}$, and if  $s=i-1$, then $\lambda \not \in \mathcal{D}_{r-1,i}$. 
  \end{itemize} 
 \end{itemize} 
  \end{itemize}
 \end{proof}


We also need the following key proposition (note that it is not exactly the converse of Proposition~\ref{1}).


\begin{Proposition}\label{mon. gen.} For every integer $r\geq 2$, for all $1\leq i \leq r$, if $\lambda \in \mathcal{D}_{r,i} \setminus \mathcal{D}_{r-1,i}$, then
$$x_\lambda =\underbrace{x_{n_{1,1}}}_{\text {first block}} \underbrace{ x_{n_{2,1}} \cdots  x_{n_{2,f(2)}} }_{\text {second block}} \cdots \underbrace{ x_{n_{r-1,1}} \cdots  x_{n_{r-1,f(r-1)}} }_{\text {{$(r-1)$-th block}}}   \underbrace{ x_{n_{r,1}} \cdots  x_{n_{r,\ell}} }_{\text {$r$-th block}},$$ 
where $0\leq \ell<f(r)$. 
\end{Proposition}

\begin{proof}  
The proof is by induction on $r$. For notational reasons, the cases (1) $i=r$, (2) $i=r-1$, and (3) $1\leq i \leq r-2 $ are slightly different; one begins by the case  $i=r,$ then by induction goes to the case $i=r-1$ and finally the case $1\leq i \leq r-2$.  The proofs are similar, so we only give the proof in the case $1\leq i \leq r-2$, assuming that for any $r$ the proposition is true for $i=r$ and $i=r-1.$ 

\medskip
For the initial case ($r=3$ and $i=1$): let $\lambda \in \mathcal{D}_{3,1} \setminus \mathcal{D}_{2,1},$ then by Lemma~\ref{43}, $\lambda$ has exactly two Durfee rectangles and nothing below. If 
$$x_\lambda=x_{n_{1,1}}x_{n_{2,1}}\cdots x_{n_{2,\ell}},$$ 
where $\ell<f(2)=n_{1,1}-1$, then $\lambda$ has only one Durfee rectangle, which is a contradiction since this would imply that $\lambda \in \mathcal{D}_{2,1}.$

If 
$$x_\lambda=x_{n_{1,1}}x_{n_{2,1}}\cdots x_{n_{2,f(2)}}x_{n_{3,1}}\cdots x_{n_{3,\ell}},$$ 
where $\ell \geq f(3)=n_{2,f(2)}-1$, then $\lambda$ would have some parts after its two first Durfee rectangles; again this contradicts the hypothesis that $\lambda$ has exactly two Durfee rectangles. Thus $x_\lambda$ is as in the proposition.

\medskip
Now assume by induction that the proposition is true for $r-1$ and all $1\leq i \leq r-3 $; note also that by the case (2) for $r-1$ we have  the proposition for   $i=r-2.$ 

 We will show it for $r$ and all $1\leq i \leq r-2$. Let $\lambda \in \mathcal{D}_{r,i} \setminus \mathcal{D}_{r-1,i}.$ By Lemma~\ref{43}, we know the possible shapes of such a partition $\lambda$. We remove the first $d'_1$ parts of $\lambda$, where $d'_1$ denotes the height of the first Durfee rectangle. Let us denote by $\mu$ the resulting partition.

\begin{itemize}
\item If $i=1$, then $\lambda$ has exactly $r-1$ Durfee rectangles and nothing after its last rectangle. Thus $\mu$ has exactly $r-2$ Durfee rectangles, i.e $\mu \in \mathcal{D}_{r-1,1} \setminus \mathcal{D}_{r-2,1}.$
\item If $2\leq i \leq r-2$, then $\lambda$ has exactly $r-i$ Durfee rectangles and $i-1$ Durfee squares of sizes $d_{r-i+1}\geq \cdots \geq d_{r-1}$. Thus $\mu$ has exactly $r-i-1$ Durfee rectangles and $i-1$ Durfee squares of sizes $d_{r-i+1}\geq \cdots \geq d_{r-1}$. Note that $S_\lambda=S_\mu$ and $m_\lambda=m_\mu$ thus we have one of the following cases:
\begin{itemize}
\item $S_\mu=\emptyset$ and $d_{r-1}\geq 2$.
\item $S_\mu \neq \emptyset$ and $m_\mu=1$.
\item $S_\mu \neq \emptyset, 2\leq m_\mu\leq i-1$ and $d_{r-i+m_{\mu}-1}\geq 2$.
\end{itemize}
\end{itemize}
So in all cases, by Lemma~\ref{43}, we have $\mu \in \mathcal{D}_{r-1,i} \setminus \mathcal{D}_{r-2,i}$, and by the induction hypothesis
$$x_\mu =\underbrace{x_{n_{1,1}}}_{\text {first block}} \underbrace{ x_{n_{2,1}} \cdots  x_{n_{2,f(2)}} }_{\text {second block}} \cdots \underbrace{ x_{n_{r-2,1}} \cdots  x_{n_{r-2,f(r-2)}} }_{\text {$(r-2)$-th block}}   \underbrace{ x_{n_{r-1,1}} \cdots  x_{n_{r-1,\ell}} }_{\text {$(r-1)$-th block}},$$ 
where $0\leq \ell<f(r-1).$
Recall that for all  $1\leq j \leq r-1$, we have $f_{r-1,i}(j)=f_{r,i}(j)=f(j).$ We add the monomial associated to the first Durfee rectangle of $\lambda$ to $x_\mu$ and we obtain:
 $$x_\lambda =\underbrace{x_{n_{1,1}}}_{\text {first block}} \underbrace{ x_{n_{2,1}} \cdots  x_{n_{2,f(2)}} }_{\text {second block}} \cdots \underbrace{ x_{n_{r-2,1}} \cdots  x_{n_{r-2,f(r-2)}} }_{\text {{$(r-2)$-th block}}}   x_{n_{r-1,1}} \cdots  x_{n_{r-1,\ell}} \  x_{i_1}\cdots x_{i_{d'_1}}.$$
 For all $1\leq j \leq f(r-1)-\ell$, we set $n_{r-1,\ell+j}:=i_{j}$ and define $f(r)=f_{r,i}(r)=n_{r-1,f(r-1)}-1.$
 
Note that on the one hand, since $\ell \geq 0$ and $d'_1$ is the height of the first Durfee rectangle of $\lambda$, we have
$$\ell+d'_1\geq d'_1 \geq n_{r-2,f(r-2)}-1=f(r-1).$$ 
This proves that the $(r-1)$-th block of $x_{\lambda}$ is $x_{n_{r-1,1}} \cdots x_{n_{r-1,f(r-1)}}$.

On the other hand, since $\ell<f(r-1)$ and $i_{f(r-1)-\ell}$ is a part of the first Durfee rectangle of $\lambda$, we have
$$\ell+d'_1 < f(r-1)+ d'_1 \leq f_{r-1}+i_{f(r-1)-\ell}-1 = f(r-1)+n_{r-1,f(r-1)}-1=f(r-1)+f(r).$$
This proves that the $r$-th block of $\lambda$ contains at most $f(r)-1$ variables, which concludes our proof.

\end{proof}

\begin{ex} Let $r=3$, $i=1$ and $\lambda=(6,5,5,4,3)$ as in Figure~\ref{fig:prop4.5}.

\begin{figure}[H]
\includegraphics[width=0.15\textwidth]{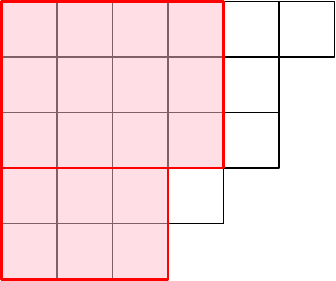}
\caption{The Young diagram of $\lambda=(6,5,5,4,3)$}
\label{fig:prop4.5}
\end{figure}
Note that this partition has exactly two Durfee rectangles and nothing below. Note also that the monomial associated to $\lambda$ is of the form:
$$x_{\lambda}=x_3x_4x^2_5x_6.$$    Take $f(1)=1$ and denote the smallest part of $\lambda$ (which is $3$) by $n_{1,1}$. By the definition of $f_{r,i}(j)=f(j)$ we have:
$$f(2)=n_{1,1}-1=3-1=2.$$
Thus, we take $n_{2,1}=4$, $n_{2,f(2)}=n_{2,2}=5$, $n_{3,1}=5$, and $n_{3,2}=6$. So we have:
$$x_{\lambda}=\underbrace{x_{3}}_{\text {(first block)}}\underbrace{x_{4}x_{5}}_{\text {(second block)}} \underbrace{x_{5}x_{6}}_{\text {(third block)}}   $$   

with $\ell=2<f(3)=n_{2,f(2)}-1=n_{2,2}-1=5-1=4.$
\end{ex}

We can now conclude this section with the proof of its main result, Theorem~\ref{thm:algebraic}.

\begin{proof}[Proof of Theorem~\ref{thm:algebraic}]
The proof is by induction on $r\geq 2,$ the case $r=2$ is immediate.
We assume that $\mathcal{C}_{r-1,i}=\mathcal{D}_{r-1,i}$ ; in particular the set of monomials in 
$\textbf{K}[x_1,x_2,\cdots]/I_{r-1,i}$ is in bijection with $\mathcal{D}_{r-1,i}.$ On the one hand,
by Proposition~\ref{1},  a monomial $x_\lambda$ of $I_{r-1,i}/I_{r,i}$ corresponds to a partition $\lambda \in\mathcal{D}_{r,i};$ such a monomial is in $I_{r-1,i},$ hence does not give rise to a monomial in $\textbf{K}[x_1,x_2,\cdots]/I_{r-1,i}$; we deduce again by the induction hypothesis that its associated partition $\lambda$ does not belong to $\mathcal{D}_{r-1,i}$. Therefore
if $x_\lambda$ is a monomial in $I_{r-1,i}/I_{r,i}$, then $\lambda \in  \mathcal{D}_{r,i} \setminus \mathcal{D}_{r-1,i}.$ On the other hand, by Proposition~\ref{mon. gen.}, we have that a monomial $x_\lambda$ associated with a partition $\lambda \in \mathcal{D}_{r,i} \setminus \mathcal{D}_{r-1,i}$ belongs to
$I_{r-1,i}/I_{r,i}$. We deduce from the exact sequence (\ref{seq}) that a basis of $\textbf{K}[x_1,x_2,\cdots]/I_{r,i}$ is given by the monomials associated with 
partitions in $\mathcal{D}_{r-1,i}$ and with the partitions in 
$\mathcal{D}_{r,i} \setminus \mathcal{D}_{r-1,i}$. This ends the proof.
   \end{proof}

\section{Proof of Theorem \ref{thm:bressoud3.3} via the Bailey lattice}
\label{sec:Bailey}

\subsection{The Bailey lattice}

Recall~\cite{AAR} that a Bailey pair $(\alpha_n,\,\beta_n)\equiv(\alpha_n,\,\beta_n)_{n\geq0}$ related to $a$ is a pair of sequences satisfying: 
\begin{equation}\label{bp}
\beta_n=\sum_{j=0}^n\frac{\alpha_j}{(q)_{n-j}(aq)_{n+j}}\;\;\;\;\forall\,n\in\mathbb{N}.
\end{equation}
The Bailey lemma describes how, from a Bailey pair, one can produce infinitely many of them.
\begin{Theorem}[Bailey lemma]\label{thm:baileylemma}
If $(\alpha_n,\,\beta_n)$ is a Bailey pair related to $a$, then so is $(\alpha'_n,\,\beta'_n)$, where
$$\alpha'_n={(\rho_1,\rho_2)_n(aq/\rho_1\rho_2)^n\over
(aq/\rho_1,aq/\rho_2)_n}\,\alpha_n$$
and
$$\beta'_n=\sum_{j=0}^n{(\rho_1,\rho_2)_j(aq/\rho_1\rho_2)_{n-j}(aq/\rho_1\rho_2)^j\over (q)_{n-j}(aq/\rho_1,aq/\rho_2)_n}\,\beta_j.$$
\end{Theorem}
In the sequel we will consider the following particular case of Theorem~\ref{thm:baileylemma}, obtained by letting $\rho_1,\rho_2\to\infty$.
\begin{Corollary}\label{coro:baileylemma}
If $(\alpha_n,\,\beta_n)$ is a Bailey pair related to $a$, then so is $(\alpha'_n,\,\beta'_n)$, where
$$\alpha'_n=a^nq^{n^2}\alpha_n\quad\mbox{and}\quad \beta'_n=\sum_{j=0}^n{a^jq^{j^2}\over (q)_{n-j}}\beta_j.$$
\end{Corollary}

In~\cite{AAR}, the following unit Bailey pair (related to $a$) is considered:
\begin{equation}\label{ubp}
\alpha_n^{(0)}={(-1)^nq^{n(n-1)/2}(1-aq^{2n})(a)_n\over (1-a)(q)_n},\qquad \beta_n^{(0)}=\delta_{n,0},
\end{equation}
and two iterations of Theorem~\ref{thm:baileylemma} applied to~\eqref{ubp} yields Watson's transformation \cite[Appendix, (III.18)]{GR}, which is a six  parameters finite extension of~\eqref{rr1} and~\eqref{rr2}.\\
Moreover, iterating $r\geq2$ times Corollary~\ref{coro:baileylemma} to the unit Bailey pair~\eqref{ubp} yields a new Bailey pair $(\alpha_n^{(r)},\,\beta_n^{(r)})$ with
$$\alpha_n^{(r)}=a^{rn}q^{rn^2}\alpha_n^{(0)}$$
and
$$\beta_n^{(r)}=\sum_{n\geq s_1\geq\dots\geq s_{r}\geq0}\frac{a^{s_1+\dots+s_r}q^{s_1^2+\dots+s_{r}^2}}{(q)_{n-s_1}(q)_{s_1-s_2}\dots(q)_{s_{r-1}-s_r}}\beta_{s_r}^{(0)}.
$$
Applying the definition~\eqref{bp} to this Bailey pair and letting $n\to\infty$ gives
\begin{equation*}
\sum_{s_1\geq\dots\geq s_{r-1}\geq0}\frac{a^{s_1+\dots+s_{r-1}}q^{s_1^2+\dots+s_{r-1}^2}}{(q)_{s_1-s_2}\dots(q)_{s_{r-2}-s_{r-1}}(q)_{s_{r-1}}}
=\frac{1}{(aq)_\infty}\sum_{j\geq0}a^{rj}q^{rj^2}(-1)^jq^{j(j-1)/2}{(1-aq^{2j})(a)_j\over (1-a)(q)_j}.
\end{equation*}
Now taking $a=1$, the right-hand side of this formula is equal to
\begin{align*}
\frac{1}{(q)_\infty}\left(1+\sum_{j\geq1}q^{rj^2}(-1)^jq^{j(j-1)/2}(1+q^j)\right)&=\frac{1}{(q)_\infty}\sum_{j\in\mathbb{Z}}(-1)^jq^{(2r+1)j^2/2}q^{j/2}\\
&=\frac{(q^{2r+1},q^r,q^{r+1};q^{2r+1})_\infty}{(q)_\infty},
\end{align*}
where the last equality follows from the Jacobi triple product identity~\cite[Appendix, (II.28)]{GR}
\begin{equation}\label{jtp}
\sum_{j\in\mathbb{Z}}(-1)^jz^jq^{j(j-1)/2}=(q,z,q/z;q)_\infty,
\end{equation}
with $q$ replaced by $q^{2r+1}$ and $z=q^r$.

 Therefore we get the $i=0$ case of~\eqref{AGP} (equivalently the $i=r$ instance of~\eqref{eq:AGri}). In the same way, one gets the $i=r-1$ case of~\eqref{AGP} (equivalently the $i=1$ instance of~\eqref{eq:AGri}) by choosing $a=q$ above.

This method is an efficient way to show these two instances of the Andrews--Gordon identities, but it fails when one aims to prove them in such a direct way for general $i$. The concept of Bailey lattice was therefore developed in~\cite{AAB} to prove~\eqref{eq:AGri} for general $i$ in a similar fashion (see also~\cite{ASW, BIS} for alternative methods avoiding the use of the Bailey lattice). In~\cite{AAB}, the authors change the parameter $a$ at some point before iterating the Bailey lemma, therefore providing a concept of Bailey lattice instead of the above classical Bailey chain.  Here is the tool proved in~\cite{AAB}.
\begin{Theorem}[Bailey lattice]\label{thm:baileylattice}
If $(\alpha_n,\,\beta_n)$ is a Bailey pair related to $a$, then  $(\alpha'_n,\,\beta'_n)$  is a Bailey pair related to $a/q$, where
$$\alpha'_0=\alpha_0, \quad \alpha'_n=(1-a)\left(\frac{a}{\rho_1\rho_2}\right)^n\frac{(\rho_1,\rho_2)_n}{(a/\rho_1,a/\rho_2)_n}\left(\frac{\alpha_n}{1-aq^{2n}}-\frac{aq^{2n-2}\alpha_{n-1}}{1-aq^{2n-2}}\right),$$
and
$$\beta'_n=\sum_{j=0}^n{(\rho_1,\rho_2)_j(a/\rho_1\rho_2)_{n-j}(a/\rho_1\rho_2)^j\over (q)_{n-j}(a/\rho_1,a/\rho_2)_n}\,\beta_j.$$
\end{Theorem}
Note that, as explained in~\cite[Theorem~3.1]{W}, both Theorems~\ref{thm:baileylemma} and~\ref{thm:baileylattice} can even be embedded in a single more general result, itself giving rise to infinitely many such lattices.

Again we will consider the special case where $\rho_1,\rho_2\to\infty$. More precisely, we will use the following consequence of Theorem~\ref{thm:baileylattice} obtained in~\cite[Corollary 4.2]{AAB} by iterating $r-i$ times Corollary~\ref{coro:baileylemma}, then using Theorem~\ref{thm:baileylattice} with $\rho_1,\rho_2\to\infty$, and finally $i-1$ times Corollary~\ref{coro:baileylemma} with $a$ replaced by $a/q$, and at the end letting $n\to\infty$.
\begin{Corollary}\label{coro:baileylattice}
If $(\alpha_n,\,\beta_n)$ is a Bailey pair related to $a$, then for all integers $0\leq i\leq r$, we have:
\begin{multline}\label{lattice}
\sum_{s_1\geq\dots\geq s_{r}\geq0}\frac{a^{s_1+\dots+s_r}q^{s_1^2+\dots+s_{r}^2-s_1-\dots-s_{i}}}{(q)_{s_1-s_2}\dots(q)_{s_{r-1}-s_r}}\beta_{s_r}=\frac{1}{(a)_\infty}\times\left(\alpha_0+\sum_{j\geq1}(1-a)a^{ij}q^{i(j^2-j)}\right.\\
\left.\left(\frac{a^{(r-i)j}q^{(r-i)j^2}\alpha_j}{1-aq^{2j}}-\frac{a^{(r-i)(j-1)+1}q^{(r-i)(j-1)^2+2j-2}\alpha_{j-1}}{1-aq^{2j-2}}\right)\right).
\end{multline}
\end{Corollary}
By applying Corollary~\ref{coro:baileylattice} to the unit Bailey pair~\eqref{ubp} with $a=q$, \eqref{eq:AGri} is proved in~\cite{AAB}, after factorizing the right-hand side by~\eqref{jtp} and replacing $i$ by $i-1$.

\subsection{Proof of Theorem~\ref{thm:bressoud3.3}}

Recall from the introduction that, omitting the dependence on $r$, the left-hand side of~\eqref{Br3.3} is written for $0\leq i\leq r-1$ as
$$S_i(q)=\sum_{s_1\geq\dots\geq s_{r-1}\geq0}\frac{q^{s_1^2+\dots+s_{r-1}^2-s_1-\dots-s_i}}{(q)_{s_1-s_2}\dots(q)_{s_{r-1}-s_{r-2}}(q)_{s_{r-1}}}.$$
We want to use Corollary~\ref{coro:baileylattice} with the unit Bailey pair~\eqref{ubp} with $a=1$ to compute $S_i(q)$. To do this, we first rewrite the right-hand side of~\eqref{lattice} by shifting the index $j$ to $j+1$ in the summation involving $\alpha_{j-1}$:
\begin{equation}\label{regroupement}
\frac{1-a^{i+1}}{(a)_\infty}\alpha_0+\frac{1-a}{(a)_\infty}\sum_{j\geq1}a^{ij}q^{i(j^2-j)}\frac{a^{(r-i)j}q^{(r-i)j^2}\alpha_j}{1-aq^{2j}}(1-a^{i+1}q^{(2i+2)j}).
\end{equation}
Now  we use $(1-a)/(a)_\infty=1/(aq)_\infty$ and take the unit Bailey pair~\eqref{ubp} with $a=1$ to derive 
$$S_i(q)=\frac{1}{(q)_\infty}\left(i+1+\sum_{j\geq1}(-1)^jq^{rj^2-ij+j(j-1)/2}\,\frac{1-q^{(2i+2)j}}{1-q^{j}}\right).$$
Expanding 
$$\frac{1-q^{(2i+2)j}}{1-q^{j}}=\left(1+q^{(i+1)j}\right)\sum_{k=0}^iq^{kj},$$
we obtain
\begin{eqnarray*}
S_i(q)&=&\frac{1}{(q)_\infty}\left(i+1+\sum_{j\geq1}\sum_{k=0}^i(-1)^jq^{(2r+1)j^2/2+(2k-2i-1)j/2}\left(1+q^{(i+1)j}\right)\right)\\
&=&\frac{1}{(q)_\infty}\left(i+1+\sum_{j\geq1}(-1)^jq^{(2r+1)j^2/2}\left(\sum_{k=0}^iq^{(2k-2i-1)j/2}+\sum_{k=0}^iq^{(2k+1)j/2}\right)\right)\\
&=&\frac{1}{(q)_\infty}\left(i+1+\sum_{j\geq1}(-1)^jq^{(2r+1)j^2/2}\left(\sum_{k=0}^iq^{(2k-2i-1)j/2}+\sum_{k=0}^iq^{(2i-2k+1)j/2}\right)\right)\\
&=&\sum_{k=0}^i\frac{1}{(q)_\infty}\sum_{j\in\mathbb{Z}}(-1)^jq^{(2r+1)j^2/2}q^{(2k-2i-1)j/2}\\
&=&\sum_{k=0}^i\frac{1}{(q)_\infty}\sum_{j\in\mathbb{Z}}(-1)^jq^{(2r+1)j(j-1)/2}q^{(k+r-i)j}.
\end{eqnarray*}
Finally, by using~\eqref{jtp} with $q$ replaced by $q^{2r+1}$ and $z=q^{k+r-i}$, we get the right-hand side of~\eqref{Br3.3}.

\section{Bijection between $\mathcal{A}_{r,r-1}$ and $\mathcal{D}_{r,r-1}$ for all $r\geq 2$}\label{sec:bij}

We conclude this paper by giving a weight-preserving bijection between $\mathcal{A}_{r,r-1}$ and $\mathcal{D}_{r,r-1}$ for all $r\geq 2$. By Theorem~\ref{th:Durfee_Bottom}, this  will provide a bijective proof of the original conjecture in the case $i=r-1$.
Recall from the end of Section~\ref{sec:outline} that the bijections between  $\mathcal{A}_{r,1}$ and $\mathcal{D}_{r,1}$ on the one hand, and between  $\mathcal{A}_{r,r}$ and $\mathcal{D}_{r,r}$ on the other hand, are almost trivial. Therefore, this is the only non-trivial case that we are able to treat bijectively.

For all integers $r \geq 2$ and $1\leq i  \leq r$, let us denote by $\mathcal{A}_{r,i}(n)$ (resp. $\mathcal{D}_{r,i}(n)$) the set of partitions of $n$ belonging to $\mathcal{A}_{r,i}$ (resp. $\mathcal{D}_{r,i}$). We have the following.

 \begin{Theorem}\label{bijection}
For all $r\geq 2,$  there is a bijection $T$ between $\mathcal{D}_{r,r-1}(n)$ and  $\mathcal{A}_{r,r-1}(n)$.
 \end{Theorem}

Our strategy for proving Theorem \ref{bijection} is the following. First, in Lemma \ref{T}, we give a bijection $T$ between some set $\mathcal{A''}_r(n)$ (that we will soon define) and $\mathcal{A}_{r,r-1}(n)$. Then, we prove in Lemma \ref{lem:A''D} that the sets  $\mathcal{A''}_r(n)$ and $\mathcal{D}_{r,r-1}(n)$ are actually the same.

\medskip

Let us first introduce some notation.
Given a partition $\lambda $ with exactly $s$ successive Durfee squares of sizes $n_1\geq \cdots \geq n_s,$ we denote by $A_{\lambda}$ the set of all indices $1\leq j \leq s$ such that there is a box directly to the right of the bottom-right corner of the $j$-th Durfee square, i.e.
 $$A_\lambda=\{1\leq j \leq s | \ \lambda_{\sum_{k=1}^{j}n_k}> n_j\}.$$
 If $A_\lambda \neq \emptyset$ we define $M_\lambda:=\max\{1\leq j \leq s \ | \ j\in A_{\lambda}\}$ and $m_\lambda:=\min\{1\leq j \leq s \ | \ j\in A_{\lambda}\}.$
 
 On the example of Figure \ref{fig:A'}, we have $A_{\lambda} = \{1, 3\}$, $M_{\Lambda}=3$, and $m_{\lambda}=1$.
 
\begin{figure}[H]
\includegraphics[width=0.4\textwidth]{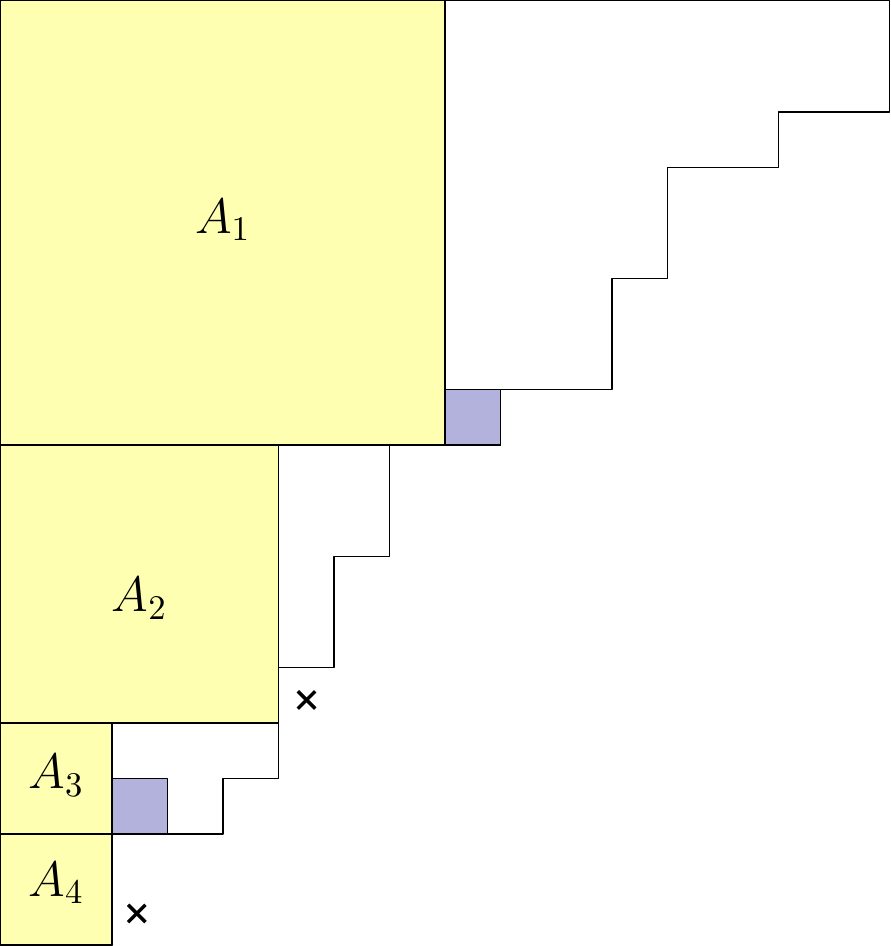}
\caption{A partition in $\mathcal{A'}_{5}(n)$}
\label{fig:A'}
\end{figure}

  Let us denote by $\mathcal{A'}_r(n)$  the set of all partitions $\lambda \in \mathcal{A}_{r,r}(n)$ with exactly $r-1$ Durfee squares and such that $A_\lambda \neq \emptyset.$ We define :
\begin{equation}
\label{eq:defA''}
\mathcal{A''}_r(n):= \mathcal{A}_{r-1,r-1}(n) \sqcup \mathcal{A'}_r(n). 
\end{equation}
Note that $\mathcal{A''}_r(n)$ is  a subset of $\mathcal{A}_{r,r}(n).$ 
   
\medskip

On the other hand, we have
\begin{equation}
\label{eq:defB}
\mathcal{A}_{r,r-1}(n)=\mathcal{A}_{r-1,r-1}(n) \sqcup \mathcal{F}_{r-1}(n),
\end{equation}
where $\mathcal{F}_{r-1}(n)$ is the set of partitions $\lambda \in \mathcal{A}_{r,r-1}(n)$ with exactly $r-2$ Durfee squares and a vertical Durfee rectangle.

Now, given a  partition $\mu \in \mathcal{F}_{r-1}(n)$ with exactly $r-2$ Durfee squares of sizes $n_1\geq \cdots \geq n_{r-2}$ and a vertical Durfee rectangle of size $(n_{r-1}+1) \times n_{r-1}$, we define $F_\mu$ to be the set of indices $1\leq j \leq r-2$  such that the part just below the $j$-th Durfee square is strictly smaller than this square, i.e.
  $$F_\mu=\{1\leq j \leq r-2| \ \mu_{\sum_{k=1}^{j} n_k+1}<n_j\}.$$
 If $F_\mu \neq \emptyset$, we define $M'_\mu:= \max\{1\leq j \leq r-2 \ | \ j\in F_{\mu}\}$. 

\medskip
By \eqref{eq:defA''} and \eqref{eq:defB}, to define a bijection between $\mathcal{A''}_r(n)$ and $\mathcal{A}_{r,r-1}(n)$, it suffices to define a bijection between $\mathcal{A'}_r(n)$ and $\mathcal{F}_{r-1}(n)$.
Thus we define a transformation $T:\mathcal{A''}_r(n) \longrightarrow \mathcal{A}_{r,r-1}(n)$ as follows.

Let $\lambda \in \mathcal{A''}_r(n).$ If $\lambda \in \mathcal{A}_{r-1,r-1}(n),$ then $T$ leaves $\lambda$ unchanged, $T(\lambda):=\lambda$. Otherwise, if $\lambda \in \mathcal{A'}_r(n)$, we know that $\lambda$ has exactly $r-1$ Durfee squares of sizes $n_1\geq \cdots \geq n_{r-1}$ and $A_\lambda \neq \emptyset$. Note that the $M_\lambda$-th square of $\lambda$ and the column to its right form a horizontal rectangle of height $n_{M_\lambda}$. To obtain $T(\lambda)$, we rotate this rectangle by $90$ degrees and we obtain a partition with $r-2$ successive Durfee squares of sizes $n_1\geq \cdots \geq n_{r-2}$ and a vertical Durfee rectangle of size $(n_{r-1}+1)\times n_{r-1}$, and no part after. Therefore $T(\lambda)$ belongs to $\mathcal{F}_{r-1}(n)$.
 In other words:
 
  $$T(\lambda)=(\lambda_1,\cdots ,\lambda_{\sum_{k=1}^{M_\lambda-1}n_k}, \lambda_{\sum_{k=1}^{M_\lambda-1}n_k+1}-1,\cdots,\lambda_{\sum_{k=1}^{M_\lambda}n_k}-1,n_{M_\lambda},\lambda_{\sum_{k=1}^{M_\lambda}n_k+1},\cdots ,\lambda_{\sum_{k=1}^{r-1}n_k}) \in \mathcal{A}_{r,r-1}(n).$$
  
  \begin{figure}[H]
\includegraphics[width=0.7\textwidth]{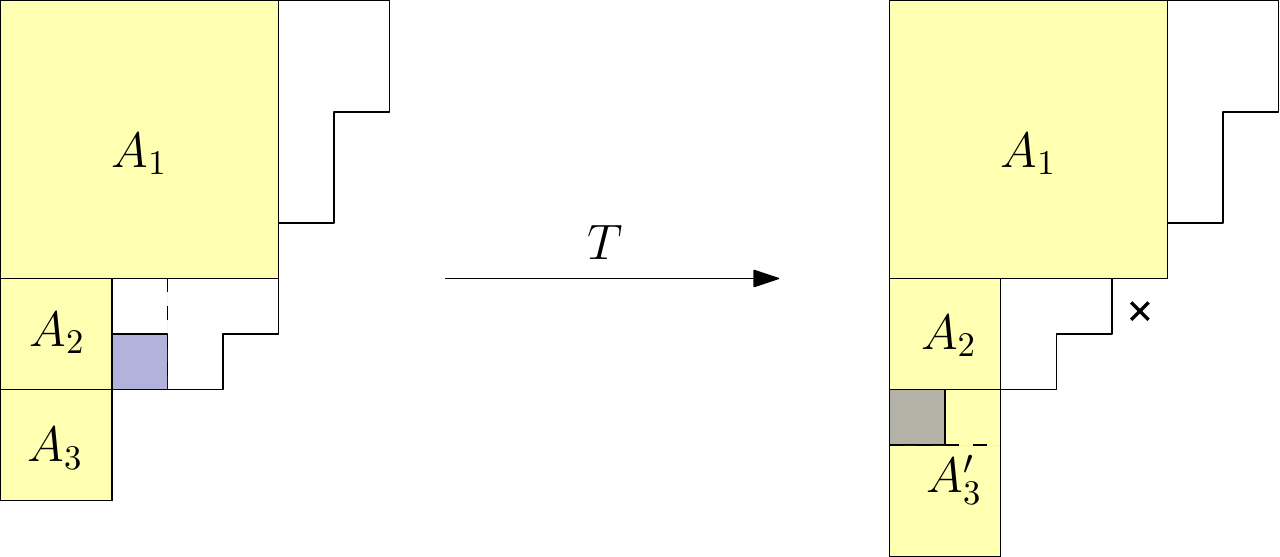}
\caption{Example of the transformation $T$ between partitions in $\mathcal{A'}_{4}(n)$ and $\mathcal{F}_{3}(n)$}
\label{fig:bijA'B}
\end{figure}
  
  We have the following result.
\begin{Lemma}\label{T} For every integer $r\geq 2$, the transformation $T$ described above is a bijection between $\mathcal{A''}_r(n)$ and $\mathcal{A}_{r,r-1}(n)$.
\end{Lemma} 
 \begin{proof}

  Let $\lambda \in \mathcal{A'}_r(n).$ Note that by definition of the transformation $T$ we have:
\begin{itemize}
\item  If $M_\lambda=r-1$, then the part just below the $(r-2)$-th Durfee square of $T(\lambda)$ is strictly less than the size of this square, and so $M'_{T(\lambda)}=r-2.$

\item  If $2\leq M_\lambda \leq  r-2,$ then for all $M_\lambda\leq j \leq r-2$, the part right after the $j$-th square of $T(\lambda)$ is equal to the size of this square and the part right after its $(M_\lambda-1)$-th square is strictly less than the size of this square. This means that $M'_{T(\lambda)}=M_\lambda-1$. In the example of Figure \ref{fig:bijA'B}, we have $M_\lambda=2$ and $M'_{T(\lambda)}=1$.
\item If $M_\lambda =1,$ then for all $1\leq j \leq r-2$ the part right after the $j$-th square of $T(\lambda)$ is equal to the size of this square and so $F_{T(\lambda)}=\emptyset.$
\end{itemize}  
  Note also that this process is reversible, and we can define the reverse transformation  $T':\mathcal{A}_{r,r-1}(n) \longrightarrow \mathcal{A''}_{r}(n)$ of $T$ as follows.
  
If $\mu \in \mathcal{A}_{r-1,r-1}(n)$, we set $T'(\mu):=\mu\in \mathcal{A''}_r(n).$  For each partition $\mu \in \mathcal{F}_{r-1}(n)$, which has exactly $r-2$ Durfee squares of size $n_1\geq \cdots \geq n_{r-2}$ and a vertical Durfee rectangle of size $(n_{r-1}+1) \times n_{r-1}$, we do the following:

\begin{itemize}
\item  If $F_\mu \neq \emptyset$  and $ M'_{\mu}=r-2,$ then we send $\mu$ to a partition $\lambda$ by removing its last part and adding a box to each of its last remaining $n_{r-1}$ parts, i.e.
  $$T'(\mu)=\lambda=(\mu_1,\cdots,\mu_{\sum_{k=1}^{r-2}n_k},\mu_{\sum_{k=1}^{r-2}n_k+1}+1,\cdots,\mu_{\sum_{k=1}^{r-1}n_k}+1)\in \mathcal{A'}_r(n). $$
  Note that in this case $M_\lambda=r-1$.
  \item  If $F_\mu \neq \emptyset$  and $1\leq M'_{\mu}\leq r-3,$ then we send $\mu$ to a partition $\lambda$ by removing the row below the $(M'_\mu+1)$-th square and adding it as a column to the right of this square, i.e.
  $$T'(\mu)=\lambda=(\mu_1,\cdots,\mu_{\sum_{k=1}^{M'_\mu}n_k},\mu_{\sum_{k=1}^{M'_\mu}n_k+1}+1,\cdots,\mu_{\sum_{k=1}^{M'_\mu+1}n_k}+1,\mu_{\sum_{k=1}^{M'_\mu+1}n_k+2},\cdots ,\mu_{\sum_{k=1}^{r-1}n_k+1})\in \mathcal{A'}_r(n). $$  
  Note that in this case $M_\lambda=M'_\mu+1.$
  \item if $F_{\mu} = \emptyset$, then we transform $\mu$ into a partition $\lambda$ by removing the part $\mu_{n_1+1}=n_1$ and adding $n_1$ boxes to the first $n_1$ parts of $\mu.$, i.e.
  $$T'(\mu)=\lambda=(\mu_1+1,\cdots \mu_{n_1}+1,\mu_{n_1+2},\cdots \mu_{\sum_{k=1}^{r-1}n_k+1} )\in \mathcal{A'}_r(n).$$
  Note that in this case $M_\lambda=1.$
\end{itemize}
 So by the definition of the transformations $T$ and $T'$, for each partition $\lambda \in \mathcal{A''}_r(n)$ and each partition $\mu \in \mathcal{A}_{r,r-1}(n)$, we have
 $T(T'(\mu))=\mu$ and $T'(T(\lambda))=\lambda.$
  This proves that the transformation $T$ is a bijection between $\mathcal{A''}_r(n)$ and $\mathcal{A}_{r,r-1}(n)$ by giving a one to one correspondence between:
  \begin{itemize}
  \item the partitions of $\mathcal{A}_{r-1,r-1}(n)$ and themselves; 
  
  \item  the partitions $\lambda$ of $\mathcal{A'}_r(n)$ with $2\leq M_\lambda\leq r-1$ and the partitions $\mu$ of $\mathcal{F}_{r-1}(n)$ with $F_\mu \neq \emptyset$ and $M'_{\mu}=M_{\lambda}-1$;
  
 \item the partitions $\lambda$ of $\mathcal{A'}_r(n)$ with $M_\lambda=1$ and the partitions $\mu$ of $\mathcal{F}_{r-1}(n)$ with $F_\mu = \emptyset$. 
  \end{itemize}
   \end{proof}

Now the only thing left to do to prove Theorem \ref{bijection} is to show that $\mathcal{A''}_r(n) = \mathcal{D}_{r,r-1}(n)$.

 \begin{Lemma}\label{lem:A''D}
For all $r\geq 2 $, we have
$$\mathcal{A''}_r(n) = \mathcal{D}_{r,r-1}(n).$$
 \end{Lemma}
  
\begin{proof}
To prove this, we will show that 
\begin{equation}\label{r-1}
\mathcal{A}_{r,r}(n)\setminus \mathcal{A''}_r(n)=\mathcal{D}_{r,r}(n) \setminus \mathcal{D}_{r,r-1}(n).
 \end{equation}
 Since we know that $\mathcal{A}_{r,r}(n)$ and $\mathcal{D}_{r,r}(n)$ are both equal to the set of partitions of $n$ with at most $r-1$ consecutive Durfee squares, and since $\mathcal{D}_{r,r-1}(n) \subset \mathcal{D}_{r,r}(n)$ and $\mathcal{A''}_r(n) \subset \mathcal{A}_{r,r}(n)$, Equation (\ref{r-1}) would prove the equality between $\mathcal{A''}_r(n)$ and $\mathcal{D}_{r,r-1}(n).$ Then by Lemma \ref{T} the transformation $T$ would define a bijection between $\mathcal{A}_{r,r-1}(n)$ and  $\mathcal{D}_{r,r-1}(n)$ as stated.
 
 We now prove~\eqref{r-1}. To do so, let $\lambda$ be a partition in $ \mathcal{A}_{r,r}(n).$ It has exactly $s$ successive Durfee squares, with $1\leq s \leq r-1$. By the definition of $\mathcal{A''}_r(n)$, the left-hand side of \eqref{r-1} is the set of partitions $\lambda$ of $n$ with exactly $r-1$ successive Durfee squares such that the last row of each Durfee square is a part of $\lambda$. i.e., $\mathcal{A}_{r,r}(n)\setminus \mathcal{A''}_r(n)$ is equal to the following set of partitions:
 
 $$\{\lambda \in \mathcal{A}_{r,r}(n) | \lambda \ \text{has exactly} \ r-1 \ \text{Durfee squares and } A_{\lambda}=\emptyset\}.$$
 
 Let us now describe the right-hand side of \eqref{r-1}. Recall that $\mathcal{D}_{r,r-1}(n)$ is the set of partitions of $n$ with at most one horizontal Durfee rectangle and $r-2$ Durfee squares. Let $\lambda$ be a partition in $\mathcal{D}_{r,r}(n)=\mathcal{A}_{r,r}(n)$, i.e. a partition of $n$ with exactly $s$ successive Durfee squares (of sizes $d_1\geq \cdots \geq d_s$) with $1\leq s \leq r-1.$ 
\begin{itemize}
\item If $A_\lambda \neq \emptyset:$ 
\\
\begin{itemize}
\item If $m_\lambda=1,$ then $\lambda$ has one horizontal Durfee rectangle of height $d_1$ and $s-1$ successive Durfee squares of sizes $d_2\geq \cdots \geq d_{s}$. So it belongs to $\mathcal{D}_{r,r-1}(n)$.

\item If $m_\lambda \geq 2,$ then to determine the first horizontal Durfee rectangle of $\lambda$, we take the first $d_1-1$ rows of the first Durfee square. Then the following $m_\lambda-2$ Durfee squares will have the same size as before, but start one row above. Next, we extend its $m_\lambda$-th square from top and right to obtain a Durfee square of size $d_{m_\lambda}+1.$ Now $\lambda$ has one horizontal Durfee rectangle of height $d_1-1$ and $s-1$ successive Durfee squares of sizes $d_2\geq \cdots \geq d_{{m_\lambda}-1} \geq d_{m_\lambda}+1\geq d_{m_{\lambda}+1}\geq \cdots \geq d_s$, and no part after this last square. So $\lambda$ belongs to $\mathcal{D}_{r,r-1}(n),$ for all $1\leq s \leq r-1.$ 
Figure \ref{fig:bij2} shows an example where $m_{\lambda}=2$.
\begin{figure}[H]
\includegraphics[width=0.65\textwidth]{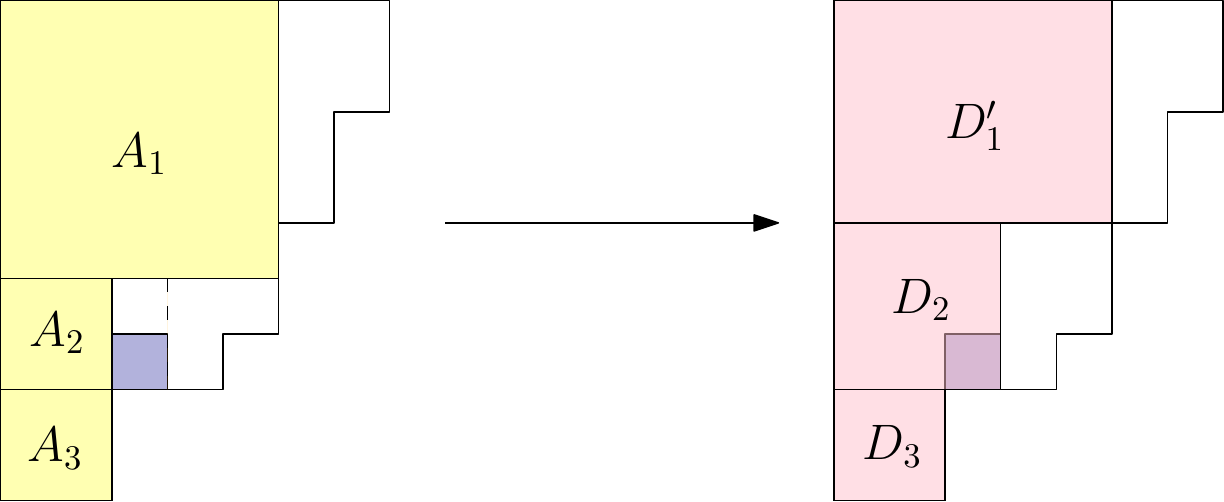}
\caption{Durfee squares and rectangles of a partition $\lambda$ with $m_{\lambda}=2$}
\label{fig:bij2}
\end{figure}

\end{itemize}

\item If $A_\lambda= \emptyset$, then to determine the first horizontal Durfee rectangle of $\lambda$, we take again the first $d_1-1$ rows of the first Durfee square. Then all the following Durfee squares have the same size but start one row above, and at the end we are left with a Durfee square of size $1$ in the last row. Now $\lambda$ has one horizontal Durfee rectangle of height $d_1-1$ and $s$ successive Durfee squares of sizes $d_2\geq \cdots \geq d_s \geq 1$. So, if $1\leq s \leq r-2$, then $\lambda$ belongs to $\mathcal{D}_{r,r-1}(n).$ This conclusion does not hold if $s=r-1$.
\end{itemize}

So $\mathcal{D}_{r,r}(n) \setminus \mathcal{D}_{r,r-1}(n)$ is also equal to the set of partitions of $n$ with exactly $r-1$ successive Durfee squares and no part after its last square such that the last row of each square is a part of the partition. This proves~\eqref{r-1}.
\end{proof}

\section{Remarks and open problems}\label{sec:final}

Our study raises a few open questions, which we list in this final section. 

First, it should be possible to show by using $q$-hypergeometric methods (or bijectively) that the sums in~\eqref{AGP} and~\eqref{eq:AGri} are the same (after changing $i$ to $r-i$ in one of them), without appealing to the product sides. This is trivially true when $i=0$ in~\eqref{AGP} (and $i=r$ in~\eqref{eq:AGri}) as both sums are the same, and easy when $i=r-1$ in~\eqref{AGP} (and $i=1$ in~\eqref{eq:AGri}) by shifting the integers $s_j$ to $s_j+1$ for all $j$ in~\eqref{AGP}. Similarly, as already mentioned in Section~\ref{sec:outline}, it could be interesting to prove analytically that~\eqref{eq:B_qserie} and~\eqref{eq:D_qserie} are equal.

Next, as discussed in Section~\ref{sec:outline}, it seems to be a challenging problem to find a bijection showing that $\mathcal{A}_{r,i} =
\mathcal{D}_{r,i}$ for all $i$ such that $1\leq i\leq r$, as we only managed to do this in the cases $i=1$, $i=r$, and $i=r-1$.

Finally, recall that in~\cite{B}, Bressoud found the counterpart for even moduli to the Andrews--Gordon identities~\eqref{eq:AGri} (replace $i$ by $r-i$ to get the form corresponding to~\eqref{eq:AGri}):
\begin{equation}\label{B}
\sum_{s_1\geq\dots\geq s_{r-1}\geq0}\frac{q^{s_1^2+\dots+s_{r-1}^2+s_{r-i}+\dots+s_{r-1}}}{(q)_{s_1-s_2}\dots(q)_{s_{r-2}-s_{r-1}}(q^2;q^2)_{s_{r-1}}}=\frac{(q^{2r},q^{r-i},q^{r+i};q^{2r})_\infty}{(q)_\infty},
\end{equation}
where $r>0$ and $0\leq i\leq r-1$ are fixed integers. Moreover there is another result, also due to Bressoud (see~\cite[(3.5)]{Br80}), which plays the role of~\eqref{Br3.3} for~\eqref{B}:
\begin{equation}\label{Br3.5}
\sum_{s_1\geq\dots\geq s_{r-1}\geq0}\frac{q^{s_1^2+\dots+s_{r-1}^2-s_1-\dots-s_i}}{(q)_{s_1-s_2}\dots(q)_{s_{r-2}-s_{r-1}}(q^2;q^2)_{s_{r-1}}}=\sum_{k=0}^{i}\frac{(q^{2r},q^{r-i+2k},q^{r+i-2k};q^{2r})_\infty}{(q)_\infty}.
\end{equation}
If, omitting the dependence on $r$, we denote by $\tilde{S}_i(q)$ the left-hand side of~\eqref{Br3.5}, then $\tilde{S}_0(q)$ gives the $i=0$ instance of~\eqref{B}, $\tilde{S}_1(q)$ yields the formula
$$\sum_{s_1\geq\dots\geq s_{r-1}\geq0}\frac{q^{s_1^2+\dots+s_{r-1}^2-s_1}}{(q)_{s_1-s_2}\dots(q)_{s_{r-2}-s_{r-1}}(q^2;q^2)_{s_{r-1}}}=2\frac{(q^{2r},q^{r-1},q^{r+1};q^{2r})_\infty}{(q)_\infty},$$
while for $2\leq i\leq r-1$, we have:
$$
\tilde{S}_i(q)-\tilde{S}_{i-2}(q)=\sum_{k=0}^{i}\frac{(q^{2r},q^{r-i+2k},q^{r+i-2k};q^{2r})_\infty}{(q)_\infty}-\sum_{k=1}^{i-1}\frac{(q^{2r},q^{r-i+2k},q^{r+i-2k};q^{2r})_\infty}{(q)_\infty}.$$
Telescoping the right-hand side of this formula and noting that the infinite products are the same for $k=0$ and $k=i$ in the above sum, we get a formula which could be seen as an even moduli counterpart of~\eqref{AGP}:
\begin{equation}\label{AGPB}
\sum_{s_1\geq\dots\geq s_{r-1}\geq0}\frac{q^{s_1^2+\dots+s_{r-1}^2-s_1-\cdots- s_i}(1-q^{s_i+s_{i-1}})}{(q)_{s_1-s_2}\dots(q)_{s_{r-2}-s_{r-1}}(q^2;q^2)_{s_{r-1}}}=2\frac{(q^{2r},q^{r-i},q^{r+i};q^{2r})_\infty}{(q)_\infty}.
\end{equation}

Therefore we are left with a natural question, namely: could the sum side of~\eqref{B} (resp.~\eqref{AGPB}) come from the computation of 
the Hilbert--Poincar\'e series of $\mathcal{R}/J'_{r,i}$ where $J'_{r,i}$ is the leading ideal of some differential ideal $\mathcal{J}'_{r,i}$ with respect to the weighted reverse lexicographical (resp. weighted lexicographical) order?

\section*{Acknowledgements}
The second and third authors are partially funded by the ANR COMBIN\'e
ANR-19-CE48-0011. The second author is funded by the SNSF Eccellenza grant number PCEFP2 202784.
We thank Jeremy Lovejoy and Ole Warnaar for their very useful comments on an earlier version of this paper. We are also grateful to the anonymous referees for their valuable suggestions which helped improve the paper.


\end{document}